\documentclass[english]{amsart}

\usepackage[OT2, T1]{fontenc}
\usepackage{amsmath}
\usepackage{amssymb}
\usepackage{amscd} 
\usepackage{tabularx}
\usepackage[all,cmtip,line]{xy}
\usepackage{amsbsy}
\usepackage{amstext}
\usepackage{amssymb}
\usepackage{esint}
\usepackage{babel}
 \usepackage{amsthm}
\usepackage{graphicx}
\usepackage{hyperref}
\usepackage{colonequals}	
\usepackage[alphabetic,lite]{amsrefs} 	

\numberwithin{equation}{subsection}
\newtheorem{theorem}[equation]{Theorem}
\newtheorem{lemma}[equation]{Lemma}
\newtheorem{proposition}[equation]{Proposition}
\newtheorem{corollary}[equation]{Corollary}
\newtheorem*{thm}{Theorem}

\theoremstyle{definition}
\newtheorem{defn}[equation]{Definition}

\theoremstyle{remark}
\newtheorem{rems}[equation]{Remark}
\newtheorem{example}[equation]{Example}

\newtheorem{para}[equation]{}

\DeclareMathOperator{\Aut}{\underline{Aut}}
\DeclareMathOperator{\End}{End}
\DeclareMathOperator{\Sym}{Sym}
\DeclareMathOperator{\Lie}{Lie}
\DeclareMathOperator{\Spec}{Spec}
\DeclareMathOperator{\gal}{gal}
\DeclareMathOperator{\SL}{SL}
\DeclareMathOperator{\Nm}{Nm}

\newcommand\g{\mathfrak g}
\newcommand\gl{\mathfrak {gl}}
\newcommand\GL{\mathrm {GL}}
\newcommand{\ncap}{\mathrm{cap}}
\newcommand{\rk}{\mathrm{rk}}
\newcommand\Q{\mathbb Q}
\newcommand\C{\mathbb C}
\newcommand\R{\mathbb R}
\newcommand\bP{\mathbb P}
\newcommand\bA{\mathbb A}

\newcommand\cO{\mathcal O}
\newcommand\cM{\mathcal M}
\newcommand\cW{\mathcal W}
\newcommand\cL{\mathcal L}
\newcommand\cN{\mathcal N}
\newcommand\cE{\mathcal E}
\newcommand\cH{\mathcal H}

\newcommand\fp{\mathfrak p}

\newcommand\yy{\boldsymbol{y}}
\newcommand\Z{\mathbb Z}
\newcommand{\tilO}{\widetilde{\Omega}}
\newcommand{\xO}{x_0}

\begin{document}

\title{Algebraic solutions of differential equations over $\mathbb{P}^{1}-\{0,\,1,\,\infty\}$}

\author {Yunqing Tang}

\address{Department of Mathematics, Harvard University}
\email{yqtang@math.harvard.edu}

\begin{abstract}
The Grothendieck--Katz $p$-curvature conjecture predicts that an arithmetic differential equation whose reduction modulo $p$ has vanishing $p$-curvatures for {\em almost all} $p,$ has finite monodromy. It is known that it suffices to prove the conjecture for differential equations on $\mathbb{P}^{1}-\{0,1,\infty\}.$ 
We prove a variant of this conjecture for $\mathbb{P}^{1}-\{0,1,\infty\},$ which asserts that if the equation satisfies a certain convergence condition for {\em all} $p,$ then its monodromy is trivial. For those $p$ for which the $p$-curvature makes sense, its vanishing implies our condition. We deduce from this a description of the differential Galois group of the equation in terms of $p$-curvatures and certain local monodromy groups.
We also prove similar variants of the $p$-curvature conjecture for the elliptic curve with $j$-invariant $1728$ minus its identity and for $\bP^1-\{\pm 1,\pm i,\infty\}$. 
\end{abstract}

\maketitle

\section{Introduction}

The Grothendieck--Katz $p$-curvature conjecture was originally raised
as a question on linear homogeneous systems of first-order differential
equations (see Conjecture (I) in \cite{K72}*{Introduction} for more details)

$$\frac{d\boldsymbol{y}}{dx}=A(x)\boldsymbol{y}.$$

Here $A(x)$ is a square matrix of rational functions of $x$ with coefficients in some number field $K$ and  $\boldsymbol{y}$ is a vector-valued function. For all but finitely many primes $\mathfrak{p}$ of $K$, it makes sense to reduce this system modulo $\mathfrak{p}$ and to define an invariant, the $p$-curvature, in terms of the resulting system. 
According to the conjecture, if \emph{almost all} (that is, all but finitely many) $p$-curvatures vanish, then the original system admits a full set of solutions in algebraic functions.

The conjecture generalizes to a smooth variety $X$ equipped with a vector bundle with an integrable connection $(M,\nabla)$ defined over some number field $K$. It 
is known that the general version of the conjecture reduces to the case when $X = \mathbb{P}_K^{1}-\{0,1,\infty\}$. (See  \cite{B01}*{2.4.1}, \cite{K82}*{Thm.~10.5}, and \cite{A05}*{7.1.4}).

In this paper, we prove a variant of the conjecture for $X = \mathbb{P}_K^{1}-\{0,1,\infty\}$ where the condition 
for almost all $\fp$ is replaced by a condition {\em for all $\fp$}. A slightly informal 
formulation of our main theorem is the following:

\begin{thm} (Theorem \ref{thm_str}) 
Let $(M, \nabla)$ a vector bundle with a connection over $X=\mathbb{P}_K^{1}-\{0,1,\infty\}$. 
If the $p$-curvature of $(M,\nabla)$ vanishes for all $\mathfrak p$, then $(M,\nabla)$ admits a full set of
rational solutions, that is, $M^{\nabla=0}$ generates $M$ as an $\mathcal{O}_{X}$-module.
\end{thm}

Let us explain the meaning of the condition of vanishing $p$-curvature at all primes $\mathfrak p$: at primes where $p$-curvature is either not defined or non-vanishing, we impose a condition on the $p$-adic radius of convergence of the horizontal sections of $(M,\nabla)$. When $(M,\nabla)$ has an integral model at a prime $\fp$ so that one can make sense of its reduction mod $p$, this convergence condition is {\em implied by} the vanishing of the $p$-curvature.

One can also extend the notion of vanishing $p$-curvature for all $\fp$ to vector bundles with connections over smooth algebraic curves equipped with a semistable model over $\cO_K$. 
However, the property of all $p$-curvature vanishing is not preserved under push-forward along finite maps from the curve in question to $\mathbb{P}^{1}-\{0,1,\infty\}$. Therefore, one cannot deduce from the above theorem that vanishing $p$-curvature for all $\fp$ implies trivial monodromy in the case of arbitrary algebraic curves. Nevertheless, when $X$ is an elliptic curve with $j$-invariant $1728$ minus its identity point, we prove:

\begin{thm}(Theorem \ref{thm_ec})
Let $X\subset \bA^2_{\Z}$ be the affine curve defined by $y^2=x(x-1)(x+1)$ and let $(M,\nabla)$ be a vector bundle with a connection over $X$. If the $p$-curvature of $(M,\nabla)$ vanishes for all $\mathfrak p$, then $(M,\nabla)$ is \'etale locally trivial. Namely, there exists a finite \'etale map $f:Y\rightarrow X$ such that $f^*(M,\nabla)$ is isomorphic to $(\cO_Y^{\mathrm{rk} M},d)$, where $d$ is the differential operator on regular functions.
\end{thm}

Unlike the previous case, passing to a finite \'etale cover is necessary. We give an example of an $(M,\nabla)$ with $G_{\gal}$ equal to $\Z/2\Z$.

Katz has shown in \cite{K82}*{Thm.~10.2} that if the $p$-curvature conjecture holds, then for any vector bundle with an integrable connection $(M,\nabla)$ on a smooth variety $X$ over $K$ as above, 
the Lie algebra $\g_{\gal}$ of the differential Galois group $G_{\gal}$ of $(M,\nabla)$ is in some sense generated by the $p$-curvatures. Namely, let $K(X)$ be the function field of $X$. The $p$-curvature conjecture implies that $\g_{\gal}$ is the 
smallest algebraic Lie subalgebra of $\gl_n(K(X))$ such that for almost all $p$ the reduction of $\g_{\gal}$ mod $p$ contains the $p$-curvature.

We use Theorem \ref{thm_str} to prove a result analogous to Katz's theorem when $X=\mathbb{P}_K^{1}-\{0,1,\infty\}$. Of course, this result (Theorem \ref{thm_fin}) involves a condition at every 
prime $\mathfrak p$, but as a compensation we describe $G_{\gal}$ and not only its Lie algebra. 
In the geometric case, namely when $(M,\nabla)$ is the relative de Rham cohomology with the Gauss--Manin connection, this extra local condition is often vacuous.
We discuss the example of the Legendre family (Remark \ref{examples}) and show that a variant of our result implies that $\g_{\gal}$ is generated by the $p$-curvatures, which recovers a result of Katz.

The main tools used to prove Theorem \ref{thm_str} and Theorem \ref{thm_ec} are the algebraicity results of Andr\'e \cite{A05}*{Thm.~5.4.3} and Bost--Chambert-Loir \cite{BCL08}*{Thm.~6.1, Thm.~7.8}.
These results generalize the classical Borel--Dwork criterion for the rationality of a formal power series. This type of results requires estimating the radius of convergence of solutions for $(M,\nabla)$ at each place of $K$.
These techniques have been used previously by Andr\'e \cite{A05}*{Sec.~6} and Bost \cite{B01}*{2.4.2} to study the 
Grothendieck--Katz conjecture in the case when the algebraic monodromy group of $(M,\nabla)$ is {\em solvable.}

The paper is organized as follows. In section \ref{sec_general} to \ref{height}, we will focus on the case when $X=\mathbb{P}_K^{1}-\{0,1,\infty\}$. In section \ref{ec}, we discuss the case when $X$ is the affine elliptic curve with $j$-invariant $1728$.

In section \ref{sec_general}, we formulate our main result, and in particular the condition which substitutes for the vanishing of the $p$-curvature when it does not make sense to reduce $(M,\nabla)$ mod $\mathfrak p$. 
We then use the main result to deduce a description of the differential Galois group following Katz. 

In section \ref{sec_alg}, we use the criterion in \cite{A05} to prove that a vector bundle with a connection $(M,\nabla),$ as in the theorem, is locally trivial for the \'{e}tale topology of $X$. 
To do this, we apply Andr\'e's criterion to the formal horizontal sections of $(M,\nabla)$ centered at a specific point $x_0.$ We obtain a lower bound for Andr\'e's analogue of their radii of convergence at archimedean places, using the uniformization of $\mathbb{P}^1_\mathbb{C}-\{0,1,\infty\}$ by the unit disc, which arises from its interpretation as the moduli space of elliptic curves with level 2 structure.  The chosen point $x_0$ corresponds to the elliptic curve with smallest stable Faltings' height and we use the Chowla-Selberg formula to deduce the lower bound.
We also discuss in this section some variants of our main theorem and an example of the Legendre family mentioned above.

In section \ref{sec_rat}, we apply the rationality criterion in \cite{BCL08}
to prove the main theorem. We give a lower bound for the local capacity of $\Omega$, the image in $\mathbb{P}^1_\mathbb{C}-\{0,1,\infty\}$
of a standard fundamental domain for $\Gamma(2)$ under the uniformization mentioned above. Together with the algebraicity of our formal solution proved in section \ref{sec_alg}, this allows us to apply the criterion in \cite{BCL08}, and deduce that the solutions of $(M,\nabla)$ are rational.

Section \ref{height} is devoted to an interpretation of our computations
in section \ref{sec_alg} in terms of the stable Faltings height, obtained by relating our estimate for archimedean places to the Arakelov degree of the restriction of the tangent bundle to some point.

In section \ref{ec}, we prove our theorem when $X$ is an affine elliptic curve with $j$-invariant $1728$ using Andr\'e's criterion and ideas in section \ref{sec_alg}. As in section \ref{sec_general}, we define the notion of $p$-curvature vanishing at bad primes using local convergence condition. Using the property of theta functions and Weierstrass-$\wp$ functions, we deduce from a result of Eremenko \cite{E} a lower bound of the archimedean radii.

In section \ref{ec_eg}, we first give an example of an $(M,\nabla)$ over the affine elliptic curve in section \ref{ec} such that its $p$-curvatures vanish for all $\fp$ but its $G_{\gal}$ is $\Z/2\Z$.
More precisely, $(M,\nabla)$ is the push-forward of $(\cO,d)$ via the degree two self-isogeny of the elliptic curve. In the second half, we discuss a variant of our main theorems when $X$ is $\bA^1-\{\pm 1,\pm i\}$ with the conclusion that $(M,\nabla)$ has finite monodromy. The proof relies on the result of Eremenko used in last section. We also give an example to show that even when $(M,\nabla)$ has good reduction everywhere and all its $p$-curvatures vanish, it can still have local monodromies of order two around the singular points $\pm 1, \pm i,\infty$.

\section*{Acknowledgement}
I thank Mark~Kisin for introducing this problem to me and all the enlightening discussions. I thank Yves~Andr\'e, Noam~Elkies, H\'el\`ene~Esnault, and Benedict~Gross for useful comments. Moreover, I am grateful to Cheng-Chiang~Tsai for conversations related to this topic and to George~Boxer, K\k{e}stutis~\v{C}esnavi\v{c}ius, Chao~Li, Andreas~Maurischat, Koji~Shimizu, Junecue~Suh, and Jerry~Wang for comments on drafts of the paper. 

\section{Statement of the main results}\label{sec_general}
Let $K$ be a number field and $\cO_K$ its ring of integers. Let $X$ be $\mathbb{P}_{\cO_K}^{1}-\{0,1,\infty\}$ and $M$ a vector bundle with a connection  $\nabla\colon\, M\rightarrow\Omega_{X_K}^{1}\otimes M$ over $X_K$. 
For a finite place $v$ of $K$ lying over a prime $p$, let $K_v$ be the completion of $K$ with respect to $v$ and denote by $\cO_v$ and $k_v$ the ring of integers and residue field of $K_v$. 
For $\Sigma$ a finite set of finite rational primes, we set $\cO_{K,\Sigma} = \cO_K[1/p]_{p \in \Sigma} \subset K.$ 

\subsection{The $p$-curvature and $p$-adic differential Galois groups}\label{pcurv}
\begin{para}\label{def_pcurv}
For $\Sigma,$ as above, sufficiently large, $(M, \nabla)$ extends to a vector bundle with connection (again denoted $(M,\nabla)$) over $X_{\cO_{K,\Sigma}}$. 
In particular, if $p \notin \Sigma$ we can consider the pull back of $(M,\nabla)$ to $X\otimes \Z/p\Z.$ 
If $D$ is a derivation on $X\otimes \Z/p\Z$, so is $D^p$. Let $\nabla(D)$ be the map $(D\otimes \mathrm{id})\circ \nabla$. Then on $X\otimes \Z/p\Z$, the \emph{$p$-curvature} 
is given by (see \cite{K82}*{Sec.~VII} for details)
\footnote{We could have defined the $p$-curvatures by considering derivations on $X_{k_v}$ for $v$ a place of $K.$ For primes which are unramified in $K,$ the two definitions are essentially equivalent, and the present definition will allow us to formulate the inequalities which arise below in a more uniform manner.
}
$$\psi_p(D)\colonequals\nabla(D^p)-\nabla(D)^p\in \mathrm{End}_{\cO_{X\otimes \Z/p\Z}}(M\otimes \Z/p\Z).$$ 
In particular, $ \psi_p \left(\frac{d}{dx}\right)=-\left(\nabla\left(\frac{d}{dx}\right)\right)^p$.
Since $\psi_p(D)$ is $p$-linear in $D$, for $X=\mathbb{P}_{\cO_K}^{1}-\{0,1,\infty\}$, the equation $\psi_p\equiv0$ is equivalent to
$ -\left(\nabla\left(\frac{d}{dx}\right)\right)^p\equiv0$.

In general, the $\psi_p$ depends on the choice of extension of $(M,\nabla)$ over $X_{\cO_{K,\Sigma}}.$ However, any two such extensions are isomorphic over $X_{\cO_{K,\Sigma'}}$ for some sufficiently large $\Sigma'.$
\end{para}

\begin{para}
Let $L$ be a finite extension of $K$ and $w$ a place of $L$ over $v$. We view $L$ as a subfield of $\C_p$ via $w.$ 
Fix an $x_0\in X(L_w).$ Given a positive real number $r$, we denote by $D(x_0,r)$ the open rigid analytic disc of radius $r,$ with center $x_0.$ Thus  
$$D(x_0,r)=\{x\in X(\C_p)\text{ such that }|x-x_0|_p<r \},$$ 
where $|\cdot|_p$ is normalized so that $|p|_p=p^{-1}$. 

Let $M^\vee$ be the dual vector bundle of $M$. It is naturally endowed with the connection such that for any local sections $m,l$ of $M$ and $M^\vee$ respectively, 
$$ d\langle l, m \rangle = \langle \nabla_{M^\vee}(l), m \rangle+\langle l, \nabla_{M}(m)\rangle .$$
\end{para}
\begin{defn} If $(V,\nabla)$ is a vector bundle with connection over some scheme or rigid space, 
we denote by $\langle V, \nabla \rangle^\otimes,$ or simply $\langle V \rangle^\otimes,$ if there is no risk of confusion regarding the connection $\nabla,$ the category of $\nabla$-stable sub quotients of all the tensor products 
$V^{\otimes m}\otimes (V^{\vee})^{\otimes n}$ for $m,n \geq 0$. If the scheme or rigid space over which $V$ is a vector bundle is connected, then this is a Tannakian category.
\end{defn}
\begin{defn}Let $F_w$ be the field of fractions of the ring of all rigid analytic functions on $D(x_0,r)$ and $\eta_w\colon \mathrm{Spec}(F_w)\rightarrow X$ the natural map. Consider the fiber functor $$\eta_w\colon \langle M|_{D(x_0,r)}\rangle ^\otimes \rightarrow \mathrm{Vec}_{F_w}; \quad V\mapsto V_{\eta_w}.$$ The \emph{$p$-adic differential Galois group} $G_w(x_0,r)$ is defined to be the automorphism group $\Aut^\otimes \eta_w$ of $\eta_w$.
\end{defn}

For $v|p$ a finite place of $K,$ we will say that $(M,\nabla)$ has {\em good reduction} at $v$ if 
$(M,\nabla)$ extends to a vector bundle with connection on $X_{\cO_v}$.
The following lemma gives the basic relation between the $p$-curvature and the $p$-adic differential Galois group.

\begin{lemma}\label{lem_lmg}
Let $x_0\in X(\cO_{L_w})$ and suppose that $(M,\nabla)$ has good reduction at $v.$
If the $p$-curvature vanishes, then the local differential Galois group 
$G_w(x_0,p^{\frac{1}{p(p-1)}})$ is trivial.
\end{lemma}
\begin{proof} To show that $G_w(x_0,p^{\frac{1}{p(p-1)}})$ is trivial, we have to show that 
the restriction of $M$ to $D(x_0,p^{\frac{1}{p(p-1)}})$ admits a full set of solutions. 
It is well known that this is the case when $\psi_p \equiv 0,$ but for the 
convenience of the reader we sketch the argument. 
See \cite{B01}*{section 3.4.2, prop. 3.9} for related arguments.

Assume there is an extension of $(M,\nabla)$ to a vector bundle with connection $(\cM, \nabla)$ 
over $X_{\cO_v}$. If $m_0$ is any section of $\cM$, then a formal section in the kernel of $\nabla$ 
is given by 
$$ m = \sum_{i=0}^{\infty} \nabla\left(\frac d {dx}\right)^i(m_0) \frac{(x-x_0)^i}{i!}(-1)^i.$$
Since $\psi_p \equiv 0$ (recall that this means the $p$-curvature vanishes on $X_{\cO_v}\otimes \Z/p\Z$), 
we have $\nabla(\frac d {dx})^p(\cM) \subset p\cM.$ Hence $\nabla(\frac d {dx})^i(m_0) \subset p^{\left[ \frac i p\right]}\cM,$ and one sees easily that the series defining $m$ converges on $D(x_0,p^{\frac{1}{p(p-1)}}).$
\end{proof}

\begin{rems}\label{rmk_M}
\leavevmode
\begin{enumerate}
\item
Unlike the notion of $p$-curvature, the definition of $G_w(x_0,r)$ does not require $(M,\nabla)$ to have good reduction. 
It depends only on the $\cO_v$-model of $X$ (which we of course always take to be 
$\mathbb P^1_{\cO_v} - \{0,1,\infty\}$),  which is used to define $D(x_0,r),$ 
but not on how $(M,\nabla)$ is extended. 
\item
If $(M,\nabla)$ has good reduction with respect to $X_{\cO_v}$ and it admits a Frobenius structure with respect to some Frobenius lifting on $X_{\cO_v}$, then $G_w(x_0,1)$ is trivial whenever $x_0\in X(\cO_v)$. See for example \cite{K10}*{17.2.2, 17.2.3}.
\end{enumerate}
\end{rems}

From now on we set $x_0 = \frac{1+\sqrt{3}i}{2},$ which corresponds to the elliptic curve with smallest stable Faltings height. In section \ref{height}, we will give a theoretical explanation of why this choice gives the best possible estimates. 
We set $G_w = G_w\big(\frac{1+\sqrt{3}i}{2}, p^{-\frac{1}{p(p-1)}}\big),$ and 
we take $L$ to be a number field containing $K(\sqrt{3}i).$

By Lemma \ref{lem_lmg}, the local differential Galois group $G_w$ is trivial when the vector bundle with connection $(M,\nabla)$ has good reduction over $v$, and $\psi_p \equiv 0$. This motivates the following definition: 
\begin{defn}\label{def_allp}
We say that {\em the $p$-curvatures of $(M,\nabla)$ vanish for all $p$} if 
\begin{enumerate}
\item $\psi_p \equiv 0$ for all but finitely many $p$,
\item $G_w= \{1\}$ for all primes $w$ of $L$.
\end{enumerate}
\end{defn}

By what we have just seen, for all but finitely many $p,$ the condition (1) makes sense, and implies (2). 
Thus (2) is only an extra condition at finitely many primes.
As above, the definition does not depend on the extension of $(M,\nabla)$ to $X_{\cO_{K,\Sigma}}$ or the choice 
of primes $\Sigma.$

\subsection{The main theorem and a Tannakian consequence}

\begin{theorem}\label{thm_str}
 Let $(M,\nabla)$ be a vector bundle with a connection over $X_K=\mathbb{P}_K^{1}-\{0,\,1,\,\infty\},$ 
 and suppose that the $p$-curvatures of $(M,\nabla)$ vanish for all $p.$
Then $(M,\nabla)$ admits a full set of rational solutions. 
\end{theorem}

The proof of this theorem is the subject of sections \ref{sec_alg}, \ref{sec_rat}. 
\begin{rems} 
By varying the conditions on the radii of convergence in (2), one can prove variants of Theorem \ref{thm_str}, whose conclusion is that $(M,\nabla)$ has finite monodromy. See Remark \ref{examples} for details.

Andr\'e has pointed out that, if one replaces (2) in Definition \ref{def_allp} by the condition that the so called {\em generic radii} of all formal horizontal sections of $(M,\nabla)$ are at least $p^{-\frac 1{p(p-1)}}$, then the analogue of Theorem \ref{thm_str} admits an easier proof. Indeed if $w|p,$ and the $w$-adic {\em generic radius} is at least  $p^{-\frac 1{p(p-1)}},$ then by \cite{BS}*{Sec. IV}, $p$ 
cannot divide the  (finite by (1) and Katz's theorem \cite{K70}*{Thm. 13.0}) order of the local monodromies. If this condition holds for all $w,$ then the 
local monodromies around $0,1,\infty$ are all trivial and hence the global monodromy is trivial.  

Once one uses (1) to show that the local monodromies are finite, this argument is `prime by prime'. We do not know if Theorem \ref{thm_str} admits a similar proof, which avoids global arguments, although this seems to us unlikely. In any case, our method allows us to deal with some cases when $X$ is an affine elliptic curve or the projective line minus more than three points. See Theorem \ref{thm_ec} and Proposition \ref{5pts}. The conclusion of both results is that $(M,\nabla)$ has finite monodromy and we will give examples in section \ref{ec_eg} with nontrivial monodromy. It seems unlikely that these results can be proved with a `prime by prime' argument.
\end{rems}

Applying Lemma \ref{lem_lmg}, we have the following corollary:

\begin{corollary}\label{main}
If $(M,\nabla)$ is defined over $X_{\Z}$ and the $p$-curvature vanishes for all primes, then $(M,\nabla)$ admits a full set of
rational solutions.
\end{corollary}

\begin{para}\label{def_G}
As in \cite{K82}, we can use our main theorem to give a description of the differential Galois group of any vector bundle with a connection $(M,\nabla)$ over $X_K$. 

Let $K(X)$ be the function field of $X_K$. Let $\omega$ be the fibre functor on $\langle M \rangle ^\otimes$ given by restriction to the generic point of $X_K$. Write $G_{\gal}= \Aut^\otimes \omega \subset \mathrm{GL}(M_{K(X)})$ for the corresponding differential Galois group
(see \cite{K82}*{Ch.~IV} and \cite{A05}*{1.3, 1.4}).

Let $G$  be the smallest closed subgroup of $\GL(M_{K(X)})$ such that:
\begin{enumerate}
\item For almost all $p,$ the reduction of $\Lie G$ mod $p$ contains $\psi_p.$
\item $G\otimes F_w$ contains $G_w$ for all $w,$ where, as above, $F_w$ is 
the field of fractions of the ring of rigid analytic functions on $D\big(\xO, p^{-\frac{1}{p(p-1)}}\big)$.
\end{enumerate}

Let $\mathfrak{g}$ be the smallest Lie subalgebra of $\GL(M_{K(X)})$ such that for almost all $p,$ 
the reduction of $\mathfrak{g}$ mod $p$ contains $\psi_p.$
As proved in \cite{K82}*{Prop.~9.3}, $\mathfrak{g}$ is contained in $\Lie\, G_{\gal}$. 
Moreover, $G_w$ is contained in $G_{\gal}\otimes F_w$ by definition. Hence $G$ is a subgroup of $G_{\gal}$.
We will see from the proof of the following theorem that (in the presence of the condition (1)), 
to define $G$ we only need to impose the condition (2) at finitely many primes. 
\end{para}

\begin{theorem}\label{thm_fin}
Let $(M,\nabla)$ be a vector bundle with a connection defined over $X_K=\mathbb{P}_K^{1}-\{0,\,1,\,\infty\}$. Then $G=G_{\gal}$.
\end{theorem}

\begin{proof}
We follow the idea of the proof of Theorem 10.2 in \cite{K82}. See also \cite{A05}*{Prop.~3.2.2}.

By a theorem of Chevalley, there exists $W$ in $\langle M \rangle ^\otimes$ and a line $L'\subset W_{K(X)}$ 
such that $G$ is the intersection of $G_{\gal}$ with the stabilizer of $L'.$ Let $W'$ be the smallest $\nabla$-stable submodule of $W_{K(X)}$ containing $L'.$ Then $W'$ has a $K(X)$-basis of the form 
$\{l,\,\nabla l,\cdots,\,\nabla^{r-1} l\}$ where $l \in L',$ $r = \rk W',$ and we have written 
$\nabla^i l$ for $\nabla(\frac d {dx})^i(l).$ 
Replacing $W$ by $W'\cap W,$ we may assume that $W_{K(X)} = W'.$ Then $L = L'\cap W$ is a line bundle in $W.$ 

As above, let $\mathfrak{g}$ be the smallest algebraic Lie subalgebra of $\GL(M_{K(X)})$ such that for almost all 
$p$ the reduction of $\mathfrak{g}$ mod $p$ contains $\psi_p.$ 
Let $\Sigma$ be a finite set of primes of $\Q$ such that $(M,\nabla)$ extends to a vector bundle $\cM$ with connection $\nabla\colon\cM\rightarrow \cM\otimes \Omega_{X_{\cO_{K,\Sigma}}}$ over $X_{\cO_{K,\Sigma}},$ and $\mathfrak{g}$ mod $p$ contains $\psi_p$ for $p \notin \Sigma.$ We also assume that $\Sigma$ contains all primes $p \leq r.$ 

Let $U \subset X_{\cO_{K,\Sigma}}$ be a non-empty open subset such that $l \in L|_U,$ $L$ and $W$ extend 
to vector bundles with connection $\cL$ and $\cW$ respectively, in $\langle \cM|_U \rangle^\otimes,$ and $\{l,\,\nabla l,\cdots,\,\nabla^{r-1} l\}$ forms a basis of $\cW.$
Let $\cN\colonequals \Sym^r \cW \otimes (\det \cW^\vee)$ with the induced connection. The argument in \cite{K82} implies that for $p\notin \Sigma$, the $p$-curvature of $(\cN,\nabla)$ vanishes. 
Let $N\colonequals\cN_{X_K\cap U}$. We will use the condition (2) in the definition of $G$ to show that $G_w$ acts 
trivially on $N_{\eta_w}$. We already know this for $p \notin \Sigma,$ by Lemma \ref{lem_lmg}. Thus we will only need to use (2) for $p \in \Sigma.$ Assuming this for a moment, we can apply Theorem \ref{thm_str} to $(N,\nabla)$ and conclude that it has trivial global monodromy. Hence $G_{\gal}$ acts as a scalar on $W$. In particular, $G_{\gal}$ stabilizes $L$ so, by the definition of $L,$ $G_{\gal}=G,$ 

Let $D\colonequals D(x_0, p^{-\frac{1}{p(p-1)}}).$ 
Recall that the category $\langle M|_{D(x_0,r)}\rangle ^\otimes\otimes F_w$ 
is obtained from $\langle M|_{D(x_0,r)}\rangle ^\otimes$ by taking the same collection of objects 
and tensoring the morphisms by $F_w.$
By the definition of $L$, the group $G_w$ acts as a character $\chi$ on $L_{\eta_w}$. The morphism 
$L_{\eta_w}\rightarrow W_{\eta_w}$ is a map between $G_w$-representations. 
By the equivalence of categories between $\langle M|_{D(x_0,r)}\rangle ^\otimes\otimes F_w$
and the category of linear representations of $G_w$ over $F_w$, this morphism is a finite $F_w$-linear combination of 
maps $L|_D \rightarrow W_D$ in $\langle M|_{D(x_0,r)}\rangle ^\otimes.$ 
In other words, there are a finite number of $\nabla$-stable line bundles $W_i\subset W_D$,  with $G_w$ acting on $W_{i,\eta_w}$ as $\chi$ such that $L|_{D}\subset \sum W_i$. In particular, $l|_D=\sum a_i\cdot w_i$, where $a_i\in F_w$ and $w_i\in W_i$. Since $\sum W_i$ is $\nabla$-stable, $\nabla^n l\in \sum W_i$ and $G_w$ acts as $\chi$ on $\nabla ^n l|_D$. As $W_{\eta_w}$ is generated by $\{l,\,\nabla l,\cdots,\,\nabla^{r-1}l\}|_D$, the group $G_w$ acts as $\chi$ on 
$W_{\eta_w}$. Hence $G_w$ acts trivially on $N_{\eta_w}.$
\end{proof}

Using the same idea as in the last paragraph of the proof above, we have the following lemma which is of independent interest.

\begin{lemma} Let $H_w \subset G_{\gal}$ be the smallest closed subgroup such that 
$G_w \subset H_w\otimes_{K(X)} F_w.$
Then $H_w$ is normal in $G_{\gal}$.
\end{lemma}

\begin{proof}
We need the following fact (see \cite{A92}*{Lem.~1}):
Assume that $G$ is a algebraic group over some field $E$. Let $H\subset G$ be a closed subgroup and $V$ an $E$-linear faithful algebraic representation of $G$. Then $H$ is a normal subgroup of $G$ if for every tensor space $V^{m,n}\colonequals V^{\otimes m}\otimes (V^\vee)^{\otimes n}$, and for every character $\chi$ of $H$ over $E$, $G$ stabilizes $(V^{m,n})^\chi$, the subspace of $V^{m,n}$ where $H$ acts as $\chi$. If $G$ is connected, then these two conditions are equivalent.

We apply this result to $H_w\subset G_{\gal}$ and $V=M_{K(X)}$. Let $L\subset {V^{m,n}}$ be a line, 
and $W\subset {V^{m,n}}$ the smallest $\nabla$-stable subspace containing $L.$ 
It suffices to show that, if $H_w$ acts via $\chi$ on $L$, then $H_w$ acts via $\chi$ on $W.$ 
This shows that $(V^{m,n})^\chi$ is $\nabla$-stable, and hence that $G_{\gal}$ stabilizes  $(V^{m,n})^\chi$.

As in the proof of the theorem above, $G_w$ acts on $W$ via $\chi.$ Hence $H_w$ is contained in the subgroup 
of $G_{\gal}$ which acts on $W$ via $\chi.$
\end{proof}

\section{Algebraicity: an application of Andr\'{e}'s theorem}\label{sec_alg}
The main goal of this section is to prove a weaker version of Theorem \ref{thm_str}. Namely, that if $(M,\nabla)$ is a vector bundle with a connection over $X_K=\mathbb{P}_K^{1}-\{0,\,1,\,\infty\}$ all of whose $p$-curvatures vanish, 
then $(M,\nabla)$ admits a full set of algebraic solutions. 

\subsection{Andr\'{e}'s algebraicity criterion.}
\begin{para}\label{taylorexp}
As the coordinate ring of $X_K$ a principal ideal domain, $M$ is free. 
Hence we may view $\nabla$ as a system of first-order homogeneous differential equations.
Thus $M\cong\mathcal{O}_{X_K}^{m}$ and $\nabla(\frac{d}{d x})\boldsymbol{y}=\frac{d\boldsymbol{y}}{dx}-A(x)\boldsymbol{y}$, where $\boldsymbol{y}$ is a section of $M$, $x$ is the coordinate
of $X$, and $A(x)$ is an $m\times m$ matrix with entries in $\mathcal{O}_{X_K}=K[x^{\pm},(x-1)^{\pm}]$.

As above, we set $x_{0}=\frac{1}{2}(1+\sqrt{3}i).$ If $\yy_0 \in L^m,$ there exists $\yy \in L[[x-x_0]]^m$ 
such that $\yy(x_0) = \yy_0$ and $\nabla(\yy) = 0.$
Our goal is to show that if the $p$-curvatures of $(M,\nabla)$ vanishes for all $p,$ then $\boldsymbol{y}$
is algebraic. 
\end{para}

\begin{para}
Now let $y\in K[[x]],$ and let $v$ be a place of $K$. 
If $v$ is finite, we denote by $p$ the characteristic of the residue field.
Let $|\cdot|_v$ be the $v$-adic norm normalized so that $|p|_v = p^{-\frac{[K_v:\Q_p]}{[K:\Q]}}$ 
if $v$ is finite, and $|x|_v = |x|_{\infty}^{-\frac{[K_v:\mathbb R]}{[K:\Q]}}$ for $x \in K,$ if $v$ is archimedean, where 
$|x|_{\infty}$ denotes the Euclidean norm on $K_v.$ When there is no confusion, we will also write $|\cdot|$ for $|\cdot|_\infty$.
For a positive real number $R,$ we denote by $D_v(0,R)$ the rigid analytic $z$-disc of $v$-adic radius $R.$  
That is $D_v(0,R)$ is defined by the inequality $|z|_v < R.$
\end{para}

We first state the definition of $v$-adic uniformization and the
associated radius $R_{v}$ defined in Andr\'{e}'s paper (\cite{A05}*{Definition 5.4.1}).
\begin{defn}\label{def_r} 
\leavevmode
\enumerate 
\item For $R \in \mathbb R^+,$ a \emph{$v$-adic uniformization} of $y$ by $D_v(0,R)$
is a pair of meromorphic $v$-adic functions $g(z),\, h(z)$ on $D_v(0,R)$
such that $h(0)=0,\, h'(0)=1$ and $y(h(z))$ is the germ at $0$
of the meromorphic function $g(z)$. 
\item Let $R_v$ be the supremum of the set of positive real $R$ for which a $v$-adic uniformization of $y$ by $D_v(0,R)$ exists. We call $R_v$ the $v$-adic \emph{radius (of uniformizability)}. 
\end{defn}

\begin{para}
In order to state the algebraicity criterion, we need to introduce two
constants $\tau(y),\,\rho(y)$, which play similar roles as the global-boundedness condition in the 
Borel--Dwork rationality criterion. Let $y=\sum_{n=0}^{\infty}a_{n}x^{n}$.
We define

$$\tau(y)=\inf_{l}\limsup_{n}\sum_{v,\, p\geq l}\frac{1}{n}\sup_{j\leq n}\log^{+}|a_{j}|_{v},$$

$$\rho(y)=\sum_{v}\limsup_{n}\frac{1}{n}\sup_{j\leq n}\log^{+}|a_{j}|_{v},$$
where $\log^{+}$ is the positive part of $\log$, that is $\log^{+}(a)=\log(a)$
if $a>1$ and is zero otherwise. 
The following is a slight reformulation of Andr\'e's criterion.
\end{para}

\begin{theorem}\label{andre}
(\cite{A05}*{Theorem 5.4.3}) Let $y\in K[[x]]$ such that $\tau(y)=0$ and $ \rho(y)<\infty$.
Let $R_{v}$ be the $v$-adic radius of $y$. If $\prod_{v}R_{v}>1$,
then $y$ is algebraic over $K(x)$.
\end{theorem}

In general the $v$-adic radius $R_v$ may be infinity or zero. We refer the reader to Andr\'{e}'s paper for a precise definition of the infinite product in such situations. In our applications of this theorem, $R_v$ will always be non-zero.
We remark that we could have also used Thm.~6.1 and Prop.~5.15 of \cite{BCL08} in place of Andr\'e's Theorem.

Suppose that $y$ is a (component of a) formal solution of $(M,\nabla)$ 
as above. By \cite{A05}, Corollary 5.4.5, if the $p$-curvatures of $(M,\nabla)$ vanish for all places over a set of rational primes of density one then $\tau(y)=0$ and $\rho(y)<\infty$.  Hence, in order to prove that $y$ is the germ of an algebraic function, we only need to prove that $\prod_{v}R_{v}>1.$

\subsection{Estimate of the radii at archimedean places.}\label{sec_Rinf}
We begin with the following simple lemma.
\begin{lemma}\label{inf}
Suppose that $\phi\colon D(0,1)\rightarrow\mathbb{P}_{\mathbb{C}}^{1}-\{0,1,\infty\}$ is a holomorphic map such that $\phi(0)=x_{0}$. Then for any archimedean place $w$ of
the number field $L$ where the connection and the initial conditions $x_0$, $\boldsymbol{y}_0$
are defined, $R_{w}\geq|\phi'(0)|_w$.
\end{lemma}
\begin{proof} Let $z$ be the complex coordinate on $D(0,1).$ 
Consider the formal power series $\phi^{*}\boldsymbol{y}.$ 
The vector valued power series $\boldsymbol{g} = \phi^*\boldsymbol{y}$ is a formal solution of the differential
equations $\frac{d\boldsymbol{g}}{dz}=(\phi'(z))^{-1}A(\phi(z))\boldsymbol{g}$
which is associated to the vector bundle with connection $(\phi^{*}M,\phi^{*}\nabla)$.
Since $D(0,1)$ is simply connected, $\boldsymbol{g}$ arises from a vector 
valued holomorphic function on $D(0,1)$ which we again denote by $\boldsymbol{g}.$

Let $t = \phi'(0)z,$ and set $R = |\phi'(0)|_{\infty}.$ 
Then we may identify $D(0,1)$ with the $t$-disc $D(0,R) = D_w(0,|\phi'(0)|_w)$ and the map $\phi$ with a map 
$$\tilde\phi: D(0,R) \rightarrow \mathbb{P}_{\mathbb{C}}^{1}-\{0,1,\infty\}$$
which satisfies $\tilde\phi'(0) = 1.$ By the definition of $R_{w}$, we have $R_{w}\geq |\phi'(0)|_w.$
\end{proof}

\begin{para} Given $x_0$, the upper bound (in terms of $x_0$) of $|\phi'(0)|$ for all such $\phi$ in the above lemma has been studied by Landau and other people. Based on the work of Landau and Schottky, Hempel gave an explicit upper bound (see \cite{H}*{Thm. 4}) that can be reached when $x_0=\frac{-1+\sqrt{3}i}{2}$. For the completeness of our paper, we give some details on the computation of $|\phi'(0)|$. 
\end{para}

\begin{para} We recall the definition of $\theta$-functions and their classical relation with the uniformization of $\mathbb{P}^1_{\C}-\{0,1,\infty\}$. Following the notation of \cite{I68} and \cite{I64}, let $$\theta_{00}(t)=\sum_{n\in\mathbb{Z}}\exp(\pi in^{2}t),\,\theta_{01}(t)=\sum_{n\in\mathbb{Z}}\exp(\pi i(n^{2}t+n)), \theta_{10}(t)=\sum_{n\in\mathbb{Z}}\exp(\pi i(n+\frac{1}{2})^{2}t)$$
These series converge pointwise to holomorphic functions on $\mathcal{H},$ which we denote by the same symbols.  
\end{para}

\begin{lemma}\label{lem_theta} (\cite{I64}*{p.~243})
These holomorphic functions $\theta_{00}^4, \theta_{01}^4, \theta_{10}^4$ are modular forms of weight $2$ and level 
$\Gamma(2)$.  Moreover, there is an isomorphism from the ring of modular forms of level $\Gamma(2)$ to $\C[X,Y,Z]/(X-Y-Z)$ given by sending $\theta_{00}^4, \theta_{01}^4$ and $\theta_{10}$ to $X,Y$ and $Z$ respectively.
\end{lemma}

We need the following basic facts mentioned in \cite{I68}*{p. 180} and \cite{I64}*{p.~244} in this section and section \ref{height}:
\begin{lemma}\label{lem_deri}
\leavevmode
\begin{enumerate}
\item
Let $\eta$ be the Dedekind eta function defined by $\eta=q^{1/24}\prod(1-q^n)$, where $q=e^{2\pi i t}$. We have $2^8\eta^{24}=(\theta_{00}\theta_{01}\theta_{10})^8$. In particular, the holomorphic functions $\theta_{00},\theta_{01},\theta_{10}$ are everywhere nonzero on the upper half plane.
\item
The derivative $\lambda'(t_0)=\pi i (\frac{\theta_{00}(t_0)\theta_{10}(t_0)}{\theta_{01}(t_0)})^4$. 
\item 
The holomorphic function $\frac{1}{2}(\theta_{00}^8+\theta_{01}^8+\theta_{10}^8)$ is the weight $4$ Eisenstein form of level $\SL_2(\Z)$ with constant term $1$ in its Fourier expansion; the holomorphic function $\frac{1}{2}(\theta_{00}^4+\theta_{01}^4)(\theta_{00}^4+\theta_{10}^4)(\theta_{01}^4-\theta_{10}^4)$ is the weight $6$ Eisenstein form of level $\SL_2(\Z)$ with constant term $1$ in its Fourier expansion.
\end{enumerate}
\end{lemma}

\begin{para}\label{lambda}
 Let $\lambda=\frac{\theta_{00}^{4}(t)}{\theta_{01}^{4}(t)}\colon\mathcal{H}\rightarrow \mathbb{P}^1(\mathbb{C})$ and $t_0=\frac{1}{2}(-1+\sqrt{3}i)$. Then $\lambda:\cH\rightarrow \bP^1(\C)-\{0,1,\infty\}$ is a covering map with $\Gamma(2)$ as the deck transformation group (\cite{KCh}, VII, \S 7). 
In particular, the projective curve defined by $v^{2}=u(u-1)(u-\lambda(t))$ is an elliptic curve. Moreover, it is isomorphic to the elliptic curve $\C/(\Z+t\Z)$ (see \emph{loc. cit.}).
\end{para}

\begin{lemma}\label{lem_t0}  The map $\lambda$ sends $t_0$ to $x_0$.
\end{lemma}
\begin{proof} Since the automorphism group of the lattice $\Z+t_0\Z$, hence that of the elliptic curve $\C/(\Z+t_0\Z)$ is of order $6$, the automorphism group of the elliptic curve $v^{2}=u(u-1)(u-\lambda(t_0))$ must also be of order $6$. In particular, $\lambda$ must send $t_0$ to either $\frac{1}{2}(1+\sqrt{3}i)$ or $\frac{1}{2}(1-\sqrt{3}i)$ (the roots of $0=j(t_0)=2^8\frac{(\lambda(t_0)^2-\lambda(t_0)+1)^3}{\lambda(t_0)^2(\lambda(t_0)-1)^2}$). Moreover, from the definition of $\theta$, we can easily see that $\lambda(t_0)$ has positive imaginary part.
\end{proof}

\begin{proposition}\label{Rinf} Let $y$ be a component of
the formal solution of the differential equations. Then $R_{w}^{\frac{[L:\Q]}{[L_w:\R]}}\geq\frac{3\Gamma(1/3)^{6}}{2^{8/3}\pi^{3}}=5.632\cdots$.
\end{proposition}

\begin{proof} 
Consider the map $\lambda\circ \alpha \colon D(0,1)\rightarrow X_{\mathbb{C}}$,
where $\alpha \colon D(0,1)\rightarrow\mathcal{H}$ is a
holomorphic isomorphism such that $\alpha(0)=t_0$, that is, $\alpha \colon z\mapsto-\frac{1}{2}+\frac{\sqrt{3}i}{2}\frac{z+1}{1-z}$.
We would like to apply Lemma \ref{inf} to the map $\lambda\circ\alpha$, which maps $0\in D(0,1)$ to $x_0$ since $\lambda(t_0)=\lambda(\frac{1}{2}(-1+\sqrt{3}i))=x_{0}$ by Lemma \ref{lem_t0}.

Note that $|x_0|= |1-x_0| = 1$, so we have $|\theta_{00}(t_0)|=|\theta_{01}(t_0)|=|\theta_{10}(t_0)|$. By Lemma \ref{lem_deri}, we have 

$$|\lambda'(t_0)|=|\pi i (\frac{\theta_{00}(t_0)\theta_{10}(t_0)}{\theta_{01}(t_0)})^4|=\pi|\theta_{00}(t_0)|^4=\pi|2^8\eta^{24}(t_0)|^{1/6}.$$

We now apply the Chowla--Selberg formula (see \cite{SC}) to $\Q(\sqrt{3}i)$:
$$|\eta(t_0)|^{4}\Im(t_0)=\frac{1}{4\pi\sqrt{3}}\left(\frac{\Gamma(1/3)}{\Gamma(2/3)}\right)^{3}.$$
Then we have
$$|\lambda'(t_0)|=\pi|2^8\eta^{24}(t_0)|^{1/6}=\frac{\pi 2^{4/3}}{4\pi\sqrt{3}\Im(t_0)}\left(\frac{\Gamma(1/3)}{\Gamma(2/3)}\right)^3.$$
We get $$|(\lambda\circ \alpha)'(0)|=|\lambda'(t_0)|\cdot|\alpha'(0)|=\frac{\pi 2^{4/3}}{4\pi\sqrt{3}\Im(t_0)}\left(\frac{\Gamma(1/3)}{\Gamma(2/3)}\right)^3\cdot 2\Im(t_0)=\frac{3\Gamma(1/3)^{6}}{2^{8/3}\pi^{3}}$$ by the fact $\Gamma(1/3)\Gamma(2/3)=\frac{2\pi}{\sqrt{3}}.$
\end{proof}

\subsection{Algebraicity of the formal solutions}

\begin{proposition}\label{alg}
Let $(M,\nabla)$ be a vector bundle with a connection over $\mathbb{P}_K^{1}-\{0,1,\infty\},$ 
and assume that the $p$-curvatures of $(M,\nabla)$ vanish for all $p$.
Then $(M,\nabla)$ is locally trivial with respect to the \'{e}tale topology of $\mathbb{P}_{K}^{1}-\{0,1,\infty\}$.
\end{proposition}

\begin{proof}
Consider $\boldsymbol{y}\in L[[(x-x_0)]]$.
By Proposition \ref{Rinf}, we have $$\prod_{w|\infty} R_w\geq 5.632\cdots.$$

If $w|p$ is a finite place of $L,$ then since $G_w$ is trivial, $(M,\nabla)$ has a full set of solutions over 
$D(x_0,|p|^{\frac{1}{p(p-1)}}).$ In particular, $\boldsymbol{y}$ is analytic on $D(x_0,|p|^{\frac{1}{p(p-1)}}).$
Hence 
$$\prod_{w|p} R_w \geq  \prod_{w|p} |p|_w^{-\frac{1}{p(p-1)}}=p^{-\frac{1}{p(p-1)}}.$$ 
and 
$$\log(\prod_{w}R_{w})\geq \log 5.6325\cdots-\sum_p\frac{\log p}{p(p-1)}>0.967\cdots.$$

Applying Theorem \ref{andre}, we have that $\boldsymbol{y}$ is algebraic. Hence $(M,\nabla)$
is \'{e}tale locally trivial.
\end{proof}

\begin{rems}\label{examples}
It is possible to define $G_w$ using different radii such that the proof of the above proposition continues to hold. 
Here are two examples:

\medskip
(1) Set $G'_w:=G_w(x_0, \frac{1}{4})$ for all primes $w|2$ and $G'_w=G_w(x_0,1)$ for other $w$.  We can define $G'$ in the same way as $G$ in section \ref{def_G} but replacing $G_w$ by $G'_w$. In this situation, we have 
$\log(\prod_{w}R_{w})\geq \log 5.6325\cdots-\log 4>0.342\cdots.$
Applying the same argument as in Theorem \ref{thm_fin}, we have $\Lie G'=\Lie G_{\gal}.$ 

In particular, if $(M,\nabla)$ is a vector bundle with connection on $X_K$ such that $\psi_p \equiv 0$ for almost all $p,$ 
and $G'_w = \{1\}$ for all $w,$ then $(M,\nabla)$ has finite monodromy. This result cannot be proved `prime by prime' because the condition at $w|2$ is too weak to imply that $2$ does not divide the order of the local monodromies. 

The equality $\Lie G'=\Lie G_{\gal}$ fails in general, if one drops  condition (1) in section \ref{def_G}, 
and defines $G'$ using just the analogue of condition (2) (that is with $G_w$ replaced by $G'_w$).
(The condition (1) is used to guarantee the assumption that $\tau(y)=0, \rho(y)<\infty$ in Theorem \ref{andre}.)

To see this, we consider the Gauss--Manin connection on $H^1_{\mathrm{dR}}$ of the Legendre family of elliptic curves. Since the Legendre family has good reduction at primes $w\nmid 2,$ $H^1_{\mathrm{dR}}$ admits a Frobenius structure at such primes, so that $G_w=\{1\}$ (see Remark \ref{rmk_M}). For $w|2$ we have 
$G_w\big(\xO,\frac{1}{4}\big)=\{1\}$ by a direct computation: as in  section \ref{sec_KS} below, we see that the matrix giving the connection lies in $\frac{1}{2}\End(M_{\cO_K})\otimes \Omega^1_{X_{\cO_K}}$ and a formal horizontal section of a general differential equation of this form will have convergence radius $\frac 1{4}$. Hence, the smallest group containing all $p$-adic differential Galois groups is trivial while $\Lie G_{\gal}=\mathfrak{sl}_2$. 
In particular, $G'$ (defined with the condition (1)) is the smallest group containing almost all $\psi_p$ and we recover a special case of \cite{K82}*{thm.~11.2}.

\medskip
(2) We now consider a variant of our result when $X$ equals to $\bP^1$ minus more than three points. Let $D$ be the union of $\{0\}$ and all $8$-th roots of unity and let $X=\bA^1-D$. Let $u_0$ be one of the preimages of $x_0$ of the covering map $f:X\rightarrow \bP^1-\{0,1,\infty\}, u\mapsto x=-\frac 1 4 (u^4+u^{-4}-2)$. We may assume that the number field $L$ contains $u_0$.

We consider the following weaker version of $p$-curvature conjecture:
\begin{proposition}\label{10pts}
Let $(N,\nabla)$ be a vector bundle with connection over $X$. Assume that the $p$-curvatures vanish for almost all $\fp$ and that for any finite place $v$, all the formal horizontal sections of $(N,\nabla)$ converges over the largest disc around $u_0$ in $X_{L_w}$. Then $(N,\nabla)$ must be \'etale locally trivial.
\end{proposition}

By direct calculation, the $w$-adic distance from $u_0$ to $D$ is $|2|_w^{\frac 1 4}$ when $w$ is finite. Then our assumption means that all the formal horizontal sections of $(N,\nabla)$ centered at $u_0$ converge over $D(u_0,|2|_w^{\frac 1 4})$. 

\begin{proof}[Proof of the proposition]
By applying Theorem \ref{andre} to the formal horizontal sections around $u_0$, one only need to show that $\prod_{w|\infty}R_w\geq 2^{1/4}$. Since the uniformization $\lambda\circ \alpha: D(0,1)\rightarrow \bP^1(\C)-\{0,1,\infty\}$ factors through $f:\bA^1(\C)-D\rightarrow \bP^1(\C)-\{0,1,\infty\}$, then for the formal horizontal sections of $(N,\nabla)$, we have $R_w\geq |5.632\cdots|_w/|f'(u_0)|_w$ by the chain rule and Lemma \ref{inf}. A direct computation shows that $\prod_{w|\infty}|f'(u_0)|_w=4$ and then $\prod_{w|\infty}R_w\geq 2^{1/4}$ by the fact $5.6325...>4\cdot 2^{1/4}$.
\end{proof}

We now formulate another possible proof of this proposition.
The idea is to reduce the problem for $(N,\nabla)$ over $X$ to $f_*(N,\nabla)$ over $\bP^1-\{0,1,\infty\}$. Over $\bP^1-\{0,1,\infty\}$, the assumption on $(N,\nabla)$ shows that for $f_*(N,\nabla)$, the $p$-adic differential group $G'_w\colonequals G_w(x_0,1)=1$ for $w\nmid 2$. Although for $w|2$, the $2$-adic differential group $G'_w\colonequals G_w(x_0, 2^{-9/4})$ is not trivial, we still have $R_w\geq |2|_w^{9/4}$ by considering the uniformization $h(z)=-\frac 1 4 ((\frac z 4+u_0)^4+(\frac z 4+u_0)^{-4}-2)$. More precisely, by the assumption on $(N,\nabla)$, we can take $R=|4|_w\cdot |2|_w^{1/4}$ in Definition \ref{def_r} and check that $|h'(0)|_w=1$ and $h(0)=x_0$. Then we apply Andr\'e's theorem and conclude that $f_*(N,\nabla)$ and hence $(N,\nabla)$ admit a full set of algebraic solutions. 

If one replaces the assumption in Proposition \ref{10pts} by that the generic radii of all formal horizontal sections of $(N,\nabla)$ are at least $|2|_w^{\frac 14}$ for all $w$ finite, the results in \cite{BS} does not apply directly due to the fact that the points in $D$ are too close to each other in $L_w$ when $w|2$. However, one may modify the argument there, especially a modified version of eqn. (3) in \emph{loc. cit.}, to see that the condition on generic radii would imply trivial monodromy of $(N,\nabla)$. 
\end{rems}

\section{Rationality: an application of a theorem of Bost and Chambert-Loir}\label{sec_rat}

In this section, we will first review the rationality criterion due
to Bost and Chambert-Loir for an algebraic formal function using capacity
norms. Then we will use the moduli interpretation of $X$ to compute the
capacity norm and verify that in our situation this theorem is applicable. 

\subsection{Review of the rationality criterion}

We will review the definition of ad\'{e}lic tube adapted
to a given point, the definition of capacity norms for the special 
case we need, and the rationality criterion in \cite{BCL08}. 
\begin{defn}(\cite{BCL08}*{Definition 5.16}) Let $Y$ be a smooth projective 
curve over $K,$ and let $(x_{0})$ be the divisor
corresponding to a given point $x_{0} \in Y(L)$ for some number field
$L \supset K$. For each finite place $w$ of $L$, let $\Omega_{w}$ be a rigid analytic open subset
of $Y_{L_{w}}$ containing $x_0$. For each archimedean place $w$, we choose one embedding $\sigma: L\rightarrow \C$ corresponding to $w$ and we let $\Omega_{w}$ be an analytic open set
of $Y_\sigma(\C)$ containing $x_0$. The collection $(\Omega_{w})$ is \emph{an ad\'{e}lic tube}
adapted to $(x_{0})$ if the following conditions are satisfied:

\begin{enumerate}
\item
for an archimedean place, the complement of $\Omega_{w}$ is non-polar (e.g. a finite collection
of closed domains and line segments); if $w$ is real, we further assume that $\Omega_{w}$ is stable under complex conjugation.
\item
for a finite place, the complement of $\Omega_{w}$ is a nonempty
affinoid subset;
\item
for almost all finite places, $\Omega_{w}$ is the tube of the specialization of $x_{0}$ in the special fiber of $Y.$ That is, $\Omega_w,$ is the open unit disc with center at $x_{0}$.
\end{enumerate}
We call $(\Omega_w)$ a \emph{weak} ad\'elic tube if we drop the condition that $\Omega_w$ is stable under complex conjugation when $w$ is real.
\end{defn}

\begin{para} Now let $Y = \mathbb{P}^1_{\cO_K}.$ 
The weak ad\'{e}lic tube that we will use can be described as follows:
\begin{enumerate}
\item For an archimedean place, $\Omega_{w}$ will be an open simply connected
domain inside $\mathbb{P}_{\mathbb{C}}^{1}-\{0,1,\infty\}$. 
\item
For a finite place, $\Omega_{w}$ will be chosen to be an open disc
of form $D(x_{0},\rho_{w})$.
\item
For almost all finite places, $\rho_{w}=1$.
\end{enumerate}
\end{para}

\begin{para}\label{cap}
For $\Omega_w$ as above, Bost and Chambert-Loir have defined the local 
capacity norms $||\cdot||_{w}^{\ncap}$ (see \cite{BCL08}*{Chapter 5}).  
These are norms on the line bundle $T_{x_{0}}X$ 
over $\Spec(\mathcal{O}_{L})$. 
The Arakelov degree of $T_{x_{0}}X$ with respect to these norms plays
the same role as $\log(\prod R_{w})$ in section \ref{sec_alg}. This degree can
be computed as a local sum after choosing a section of this bundle.
We will use the section $\frac{d}{dx}$, in which case one has the following simple description of local capacity norms:

\begin{enumerate}
\item
For an archimedean place, let $\phi\colon D(0,R)\rightarrow\Omega_{w}$
be a holomorphic isomorphism that maps $0$ to $x_{0}$, then $||\frac{d}{dx}||_{w}^{\ncap}=|R\phi'(0)|^{-1}_w$ (see \cite{B99}*{Example 3.4}). 
\item
For a finite place, $||\frac{d}{dx}||_{w}^{\ncap}=\rho_{w}^{-1}$ (see \cite{BCL08}*{Example 5.12}.
\end{enumerate} 
\end{para}

Now, we can state the rationality criterion:
\begin{theorem}\label{rat}
(\cite{BCL08}*{Theorem 7.8}) Let $(\Omega_w)$be an ad\'elic tube adapted to $(x_0)$. Suppose $y$ is
a formal power series over $X$ centered at $x_{0}$ satisfying the following
conditions:
\begin{enumerate}
\item 
For all $w$, $y$ extends to an analytic meromorphic function on $\Omega_{w}$;
\item
The formal power series $y$ is algebraic over the function field $K(X).$
\item
The Arakelov degree of $T_{x_{0}}X$ defined as $\displaystyle\sum_{w}-\log(||\frac{d}{dx}||_{w}^{\ncap})$
is positive.
\end{enumerate}
Then $y$ is rational.
\end{theorem}

\begin{corollary}\label{vrat}
The theorem still holds if we only assume that $(\Omega_w)$ is a weak adelic tube.
\end{corollary}
\begin{proof} The idea is implicitly contained in the discussion in \cite{B99}*{section 4.4}. We only need to prove that $y$ is rational over $X_{L'},$ where $L'/L$ is a finite extension which we may assume does not have any real places. 
Let $w$ be a place of $L$ and $w'$ a place of $L'$ over $w$. 

For $w$ is archimedean, choose the embedding $\sigma'\colon L'\rightarrow \C$ corresponding to $w'$ which extends 
the chosen embedding $\sigma:L\rightarrow \C$ corresponding to $w.$ We have a natural identification $Y_{\sigma'}(\C)=Y_{\sigma}(\C),$ and we take $\Omega_{w'}\colonequals\Omega_w.$  If $w$ is a finite place, we set 
$\Omega_{w'} = \Omega_w\otimes_{L_w}L_{w'}.$

Since $L'$ does not have any real places, the weak ad\'elic tube $(\Omega_{w'})$ is an ad\'elic tube. The first two conditions in Theorem \ref{rat} still hold and the Arakelov degree of $T_{x_0}X$ with respect to $(\Omega_w')$ is the same as that of $T_{x_0}X$ with respect to $(\Omega_w)$. We can apply Theorem \ref{rat} to $y$ over $X_{L'}$ and conclude that $y$ is rational.
\end{proof}

\subsection{Proof of the main theorem}\label{sec_pf}
Let $y$ be the algebraic formal function which is one component of
the formal horizontal section $\boldsymbol{y}$ of $(M,\nabla)$ over
$X_K$. 

\begin{lemma}\label{gb}
Let $y$ be as above. Then this formal power series
centered at $x_{0}$ has convergence radius equal to $1$ for almost
all finite places.
\end{lemma}
\begin{proof}
Since the covering induced by $y$ is finite \'etale over $X_L,$ by Proposition \ref{alg}, 
it is \'etale over $X_{\cO_w}$ at $x_0$ for almost all places. For such places, we have $\rho_w=1$ by lifting criterion for \'etale maps.
\end{proof}

\begin{para}\label{atube}
We now define an ad\'elic tube ($\Omega_w$) adapted to $x_0$. For an archimedean place $w$, we choose the embedding $\sigma:L\rightarrow \C$ corresponding to $w$ such that $\sigma(x_0) = (1+\sqrt{3}i)/2$.  Let $\tilO$ be the open region in the upper half plane cut out by the following six edges (see the attached figure): $\Re t=-\frac{3}{2}$, $|t+2|=1$, $|t+\frac{2}{3}|=\frac{1}{3}$, $|t+\frac{1}{3}|=\frac{1}{3}$, $|t-1|=1$, and $\Re t=\frac{1}{2}$. This is a fundamental domain of the arithmetic
group $\Gamma(2)\subset \SL_{2}(\mathbb{Z})$. 

We define $\Omega_w$ to be $\lambda(\tilO)$. 

\begin{figure}
 \centering
 \includegraphics[width=\textwidth,natwidth=610,natheight=642]{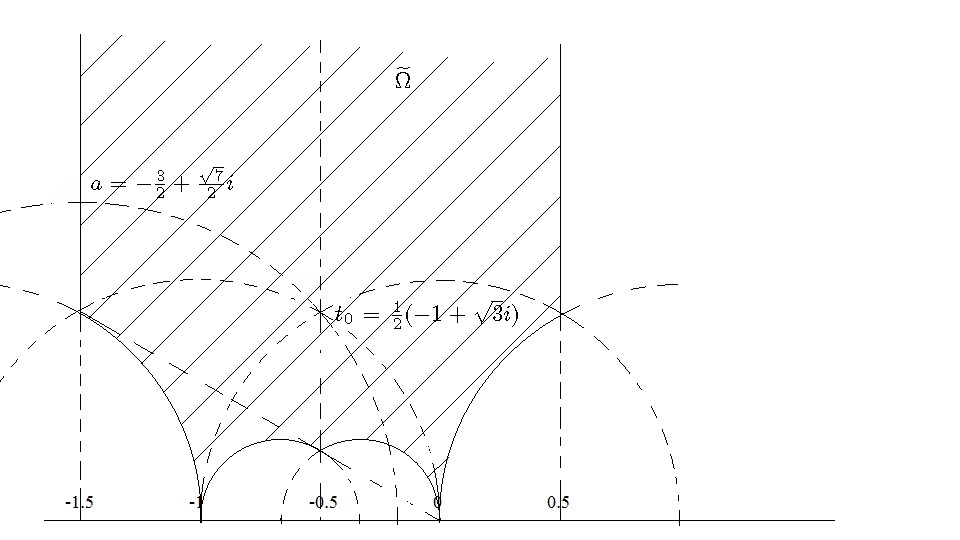}
 \end{figure}
For $w$ finite, we choose $\Omega_w$ to be $D(x_0,1)$ if $y$ is \'etale over $X_{\cO_w}$ at $x_0$; otherwise, we choose $\Omega_w$ to be $D(x_{0},p^{-\frac{1}{p(p-1)}})$.

The collection $(\Omega_w)$ is a weak ad\'elic tube and $y$ extends to an analytic (in particular meromorphic) function on each $\Omega_w$ by Lemma \ref{gb}, Lemma \ref{inf}, and Lemma \ref{lem_lmg}.
\end{para}

\begin{lemma}\label{aRinf}
The Arakelov degree of $T_{x_0}X$ with respect to the ad\'{e}lic
tube $(\Omega_{w})$ defined above is positive.
\end{lemma}

\begin{proof}
We want to give a lower bound of $(||\frac{d}{dx}||^{\ncap}_w)^{-1}$, the capacity of $\Omega_w$. Let $a=-\frac{3}{2}+\frac{\sqrt{7}}{2}i$. On the line $\Re(t)=-\frac{3}{2}$, the point $a$ is the closest point to $t_{0}=\frac{1}{2}(-1+\sqrt{3}i)$  with respect to Poincar\'{e} metric. 
The stabilizer of $t_0$ in $\SL_2(\Z)$ has order $3,$ and permutes the geodesics 
$\Re t=-\frac{3}{2}$, $|t+\frac{2}{3}|=\frac{1}{3}$, $|t-1|=1,$ and this action preserves the Poincar\'{e} metric. 
Using this, together with the fact that the distance to $t_0$ is invariant under 
$z \mapsto -1 - \bar z,$ one sees that the distance from any point on the boundary of $\tilO$ to $t_0$ is at least that from $a$ to $t_0$. 
Since $\alpha\colon D(0,1)\rightarrow \cH$ (defined in the proof of Prop. \ref{Rinf}) preserves the Poincar\'e metrics,
$\alpha^{-1}(\tilO)$ contains a disc with respect to the Poincar\'e radius equal to the distance from $t_0$ to $a$.

In $D(0,1)$, a disc with respect to Poincar\'e metric is also a disc in the Euclidean sense. Hence $\alpha^{-1}(\tilO)$ contains a disc of Euclidean radius 
$$|\alpha^{-1}(a)|=|(a-t_{0})/(a-\bar{t}_{0})|=0.45685\cdots.$$ 
Since $\lambda$ maps the fundamental domain $\tilO$ isomorphically onto $\Omega_w,$
by \ref{cap}, the local capacity $(||\frac{d}{dx}||^{\ncap}_w)^{-1}$ is at least 
$|(a-t_{0})/(a-\bar{t}_{0})|\cdot|\lambda'(\frac{1}{2}(-1+\sqrt{3}i))|$.

By \ref{cap}, we have $-\log(||\frac{d}{dx}||_{w}^{\ncap})\geq -\frac{\log p}{p(p-1)}$ when $w|p$.
Recall in Proposition \ref{Rinf} we have $|\lambda'(\frac{1}{2}(-1+\sqrt{3}i))|=5.632\cdots$, hence the Arakelov degree of $T_{x_0}X$ is
$$\sum_{w}-\log(||\frac{d}{dx}||_{w}^{\ncap})>\log(5.6325\cdots \times 0.45685\cdots)-\sum_{p}\frac{\log p}{p(p-1)}>0.184\cdots.$$
\end{proof}

Now we are ready to prove Theorem \ref{thm_str}:
\begin{proof} Applying Proposition \ref{alg}, we have
a full set of algebraic solutions $\boldsymbol{y}$. Choosing the weak ad\'{e}lic
tube as in \ref{atube} and applying Corollary \ref{vrat} 
(the assumptions are verified by \ref{atube} and Lemma \ref{aRinf}), we have that these algebraic solutions are actually rational. 

This shows that $(M,\nabla)$ has a full set of rational solutions over $X_L.$ 
Since formation of $\ker(\nabla)$ commutes with the finite extension of scalars $\otimes_KL,$ 
this implies that $(M,\nabla)$ has a full set of rational solutions over $X_K.$
\end{proof}

\section{Interpretation using the Faltings height }\label{height}

In this section, we view $X_{\mathbb{Z}[\frac{1}{2}]}$ as the
moduli space of elliptic curves with level~$2$ structure. Let  $\lambda_{0} \in X(\bar \Q)$ and $E$ the corresponding elliptic curve. Using the Kodaira--Spencer map, we will relate the Faltings height of $E$ with our lower bound for the product of radii of uniformizability (see section \ref{sec_alg}) at archimedean places of the formal solutions in $\widehat{\mathcal{O}}_{X_{K},\lambda_{0}}$. We will focus mainly on the case when $\lambda_0\in X(\bar \Z)$ and sketch how to generalize to $\lambda_{0} \in X(\bar \Q)$ at the end of this section. In this section, unlike the previous sections, we will use $\lambda$ as the coordinate of $X$.

\subsection{Hermitian line bundles and their Arakelov degrees}\label{sec_hlb}
\begin{para}
Let $K$ be a number field, and $\mathcal{O}_{K}$ its ring of integers.
Recall that an \emph{Hermitian line bundle} $(L, ||\cdot||_{\sigma})$ over $\Spec(\mathcal{O}_{K})$ is 
a line bundle $L$ over $\Spec(\mathcal{O}_{K}),$ together with 
an Hermitian metric $||\cdot||_{\sigma}$ on $L\otimes_{\sigma}\mathbb{C}$
for each archimedean place $\sigma\colon K\rightarrow\mathbb{C}.$

Given an Hermitian line bundle $(L,||\cdot||_{\sigma})$, its (normalized)
\emph{Arakelov degree} is defined as: $$\widehat{\deg}(L)\colonequals\frac{1}{[K:\mathbb{Q}]}\left(\log(\#(L/s\mathcal{O}_{K}))-\sum_{\sigma:K\rightarrow\mathbb{C}}\log||s||_{\sigma}\right),$$
where $s$ is any section.

For a finite place $v$ over $p$, the integral structure of $L$ defines a norm $||\cdot||_{v}$ on $L_{K_{v}}$. More precisely, if $s_{v}$ is a generator of $L_{\mathcal{O}_{K_{v}}}$ and $n$ is an integer, we define $||p^n s_{v}||_{v}=p^{-n[K_v:\Q_p]}$. We obtain a norm on $\cO_v$ by viewing it as the trivial line bundle. We will use $||\cdot||_v$ for the norms on different line bundle as no confusion would arise. We
may rewrite the Arakelov degree using the $p$-adic norms: $$\widehat{\deg}(L)=\frac{1}{[K:\mathbb{Q}]}\left(-\sum_{v}\log||s||_{v}\right),$$
where $v$ runs over all places of $K$.
It is an immediate corollary of the product formula that the right hand side does not depend on the choice of $s$.
\end{para}

\begin{para}
Let $E$ be an elliptic curve over a number field $K,$ and denote by  
$e\colon \Spec K\rightarrow E$ and $f\colon E\rightarrow\Spec K $ the identity and structure map respectively. 
For each $\sigma\colon K\rightarrow\mathbb{C},$ we endow $e^{*}\Omega_{E/K}^{1}=f_{*}\Omega_{E/K}$ 
with the Hermitian norm given by  $||\alpha||_{\sigma}=(\frac{1}{2\pi}\int_{\sigma E}|\alpha\wedge\bar{\alpha}|)^{\frac{\epsilon_{\sigma}}{2}}$, where $\epsilon_{\sigma}$ is $1$ for real embeddings and $2$ otherwise.

This can be used to define the Faltings' height of $E$, which we recall precisely only in the case when $E$ has good reduction over $\cO_K.$ Denote by $f:\cE \rightarrow \Spec \cO_K$ the elliptic curve over $\cO_K$ with generic fibre $E,$ and again write $e$ for the identity section of $\cE.$ The norms $||\alpha||_{\sigma}$ make 
$e^{*}\Omega_{\cE/\Spec(\mathcal{O}_{K})}^{1}=f_{*}\Omega_{\cE/\Spec(\mathcal{O}_{K})}^{1}$ into a Hermitian line bundle, and we define the \emph{(stable) Faltings height}  by 
$$ h_F(E_{\lambda}) = \widehat{\deg}(f_{*}\Omega_{\cE/\Spec(\mathcal{O}_{K})}^{1}).$$
Notice that $h_F(E_{\lambda})$ does not depend on the choice of $K$.
Here we use Deligne's definition for convenience \cite{D}*{1.2}. This differs from 
Faltings' original definition (see \cite{F85}) by a constant $\log(\pi)$.

In general, the elliptic curve $E$ would have semi-stable reduction everywhere after some field extension. We assume this is the case and $E$ has a Neron model $f:\cE \rightarrow \Spec \cO_K$ which endows $f_{*}\Omega_{\cE/\Spec(\mathcal{O}_{K})}^{1}$ a canonical integral structure. With the same Hermitian norm defined as above, we have a similar definition of Faltings height in the general case. See \cite{F85} for details. As in the good reduction case, this definition does not depend on the choice of $K$.
\end{para}

\begin{para}\label{assump}
We will assume that $\lambda_{0}$ and $\lambda_{0}-1$ are both units at each finite place. 
Given such a $\lambda_{0}$, consider the elliptic curve $E_{\lambda_{0}}$ over $\mathbb{Q}(\lambda_{0})$
defined by the equation $y^{2}=x(x-1)(x-\lambda_0)$. Then $E_{\lambda_0}$ has good reduction at primes not dividing $2,$ 
and potentially good reduction everywhere, since its $j$-invariant is an algebraic integer. 
Let $K$ be a number field such that $(E_{\lambda_{0}})_{K}$ has good reduction everywhere. 
We denote by $\cE_{\lambda_{0}}$ the elliptic curve over $\cO_K$ with generic fiber $E_{\lambda_0}.$ 
\end{para}

\begin{para}\label{sec_TRinf}
To express our computation of radii in terms of Arakelov degrees, we endow the 
$\mathcal{O}_{K}$-line bundle $T_{\lambda_{0}}(X_{\mathcal{O}_{K}})$, the tangent bundle of $X_{\mathcal{O}_{K}}$ at $\lambda_{0}$, with the structure of an Hermitian line bundle as follows.
For each archimedean place $\sigma\colon K\rightarrow \C$, we have the universal covering 
$\lambda:\mathcal{H}\rightarrow\sigma X,$ introduced in \ref{lambda}. The $\SL_{2}(\mathbb{R})$-invariant metric
$\frac{dt}{2\Im(t)}$ on the tangent bundle of $\mathcal{H}$ induces
the desired metric on the tangent bundle via push-forward. As in the proof of Proposition \ref{Rinf}, 
our lower bound on the radius of the formal solution is $|2\Im(t_{0})\lambda'(t_{0})|^{\epsilon_{\sigma}}=||\frac{d}{d\lambda}||_{\sigma}^{-1}$,
where $t_{0}$ is a point on $\mathcal{H}$ mapping to $\lambda_{0}$. It is easy to see the left hand side does not depend on the choice of $t_0$. Under the assumptions in \ref{assump}, the tangent vector
$\frac{d}{d\lambda}$ is an $\cO_K$-basis vector for the tangent
bundle $T_{\lambda_{0}}(X_{\mathcal{O}_{K}})$, and we have 
$$\displaystyle\widehat{\deg}(T_{\lambda_{0}}X)=\frac{1}{[K:\mathbb{Q}]}(-\sum_{\sigma:K\rightarrow\mathbb{C}}\log||\frac d{d\lambda}||_{\sigma})\leq\frac{1}{[K:\mathbb{Q}]}\log(\prod_{\sigma}R_{\sigma}),$$ 
where the $R_\sigma$ are the radius of uniformization discussed in section \ref{sec_Rinf}.
\end{para}
\subsection{The Kodaira--Spencer map}\label{sec_KS}

Consider the Legendre family of elliptic curves $E \subset\mathbb{P}_{\mathbb{Z}[\frac{1}{2}]}^{2}\times X_{\mathbb{Z}[\frac{1}{2}]}$ over $X_{\mathbb{Z}[\frac{1}{2}]}$ given by $y^{2}=x(x-1)(x-\lambda).$
We have the Kodaira--Spencer map (\cite{FC}*{Ch.~III,9},\cite{K72}*{1.1}):

\begin{eqnarray}\label{KS}
KS\colon (f_{*}\Omega_{E/X_{\mathbb{Z}[\frac{1}{2}]}}^{1})^{\otimes2}\rightarrow\Omega_{X_{\mathbb{Z}[\frac{1}{2}]}}^{1},\ \alpha\otimes \beta\mapsto\langle\alpha,\nabla\beta\rangle,
\end{eqnarray}
where $\nabla$ is the Gauss--Manin connection and $\langle\cdot,\cdot\rangle$
is the pairing induced by the natural polarization.

\begin{para}
Following Kedlaya's notes (\cite{Kedlaya}*{Sec.~1,3}), we choose
$\{\frac{dx}{2y},\frac{xdx}{2y}\}$ to be an integral
basis of $H_{dR}^{1}(E/X)|_{\lambda_{0}}$ and compute the Gauss--Manin
connection: 
$$\nabla\frac{dx}{2y}=\frac{1}{2(1-\lambda)}\frac{dx}{2y}\otimes d\lambda+\frac{1}{2\lambda(\lambda-1)}\frac{xdx}{2y}\otimes d\lambda.$$
The Kodaira--Spencer map then sends $(\frac{dx}{2y})^{\otimes2}$ to
$\frac{1}{2\lambda(\lambda-1)}d\lambda$. 
\end{para}
This computation shows:

\begin{lemma}\label{ksf}
Given $v$ a finite place not lying over $2$, the Kodaira--Spencer map \eqref{KS} preserves the $\cO_v$-generators of $(f_{*}\Omega_{E/X_{\mathbb{Z}[\frac{1}{2}]}}^{1})^{\otimes2}|_{\lambda_0}$ and $\Omega_{X_{\mathbb{Z}[\frac{1}{2}]}}^{1}|_{\lambda_0}$ when $\lambda_0$ and $\lambda_0-1$ are both $v$-units.
\end{lemma}

\begin{para}
For the archimedean places $\sigma$, we consider $f_{*}\Omega_{\sigma E/\Spec \C}^{1}$ with the 
metrics $||\alpha||_{\sigma}$ defined in section \ref{sec_hlb}, and we endow $\Omega^1_{X_\Z}|_{\lambda_0}$ the Hermitian line bundle structure as the dual of the tangent bundle.

To see the Kodaira--Spencer map preserves the Hermitian norms on both sides, one may argue as follows. Notice that the metrics on $(f_{*}\Omega_{\sigma E/\Spec \C}^{1})^{\otimes 2}$ and $\Omega^1_{X_\Z}$ are $\SL_2(\R)$-invariant (see for example \cite{ZP}*{Remark~3 in Sec.~2.3}). Hence they are the same up to a constant and we only need to compare them at the cusps. To do this, one studies both sides for the Tate curve.  See for example \cite{MB}*{2.2} for a related argument and Lemma \ref{lem_deri} (2) for relation between $\theta$-functions and $\Omega^1_{X}$. 
\end{para}

Here we give another argument:
\begin{lemma}\label{ksa}
The Kodaira--Spencer map preserves the Hermitian metrics: $$ ||(\frac{dx}{2y})^{\otimes2}||_{\sigma}=||\frac{d\lambda}{2\lambda_{0}(\lambda_{0}-1)}||_{\sigma}.$$
\end{lemma}

\begin{proof}
Let $dz$ be an invariant holomorphic
differential of $\mathbb{C}/(\mathbb{Z}\oplus t_{0}\mathbb{Z})$,
where $\lambda(t_{0})=\lambda_{0}$. By the theory of the Weierstrass-$\wp$
function, we have a map from the complex torus to the elliptic curve
$$u^{2}=4v^{3}-g_{2}(t_{0})v-g_{3}(t_{0})$$ 
such that $dz$ maps to
$\frac{dv}{u}$. Here $g_2$ is the weight $4$ modular form of level $\SL_2(\Z)$ with $\frac{4\pi^4}{3}$ as the constant term in its Fourier series and $g_3$ is the weight $6$ modular form with $\frac{8\pi^6}{27}$ as the constant term. Using Lemma \ref{lem_deri} (3), we see that the right hand side has three roots: $\frac{\pi^2}{3}(\theta^4_{00}(t_0)+\theta^4_{01}(t_0)), -\frac{\pi^2}{3}(\theta^4_{00}(t_0)+\theta^4_{10}(t_0)),\frac{\pi^2}{3}(\theta^4_{10}(t_0)-\theta^4_{01}(t_0))$. Hence this curve is isomorphic to $y^2=x(x-1)(x-\lambda_0)$ via the map 
\begin{equation}\label{isom}
x=\frac{v-\frac{1}{3}\pi^{2}(\theta_{00}^{4}(t_{0})+\theta_{01}^{4}(t_{0}))}{-\pi^{2}\theta_{01}^{4}(t_{0})},\ y=\frac{u}{2(-\pi^{2}\theta_{01}^{4}(t_{0}))^{3/2}},
\end{equation}
and we have $$\frac{dx}{2y}=\pi i\theta_{01}^{2}(t_{0})\frac{dv}{u}=\pi i\theta_{01}^{2}(t_{0})dz.$$ Hence $$||(\frac{dx}{2y})^{\otimes2}||_{\sigma}=|\pi^{2}\theta_{01}^{4}(t_{0})\cdot(\frac{1}{2\pi}\int_{E(\mathbb{C})}|dz\wedge d\bar{z}|)|^{\epsilon_{\sigma}}=|\pi\theta_{01}^{4}(t_{0})\Im(t_{0})|^{\epsilon_{v}}.$$
On the other hand, using Lemma \ref{lem_deri} (2),
we have $$||\frac{d\lambda}{2\lambda_{0}(\lambda_{0}-1)}||_{\sigma}^{1/\epsilon_{\sigma}}=\bigg|\frac{2\Im(t_{0})|\lambda'(t_{0})|}{2\lambda_{0}(\lambda_{0}-1)}\bigg|=\bigg|\frac{\Im(t_{0})\pi\theta_{00}^{4}(t_{0})\theta_{10}^{4}(t_{0})}{\theta_{01}^4(t_0)\lambda_{0}(\lambda_{0}-1)}\bigg|=|\pi\theta_{01}^{4}(t_{0})\Im(t_{0})|.$$
\end{proof}

\begin{proposition}\label{lem_2}
If $\lambda_0$ and $\lambda_0-1$ are both units at every finite places, we have 
 $\widehat{\deg}(T_{\lambda_0}X)= -2h_F(E_{\lambda_0})+\frac{\log 2}{3}$.
\end{proposition}

\begin{proof}
By lemma \ref{ksf} and lemma \ref{ksa}, we have 

\begin{equation}\label{eqn_ht}
\begin{split}
-\widehat{\deg}(T_{\lambda_0}X) =&\widehat{\deg}(\Omega^1_{X_{O_K}}|_{\lambda_0})\\
=&\frac{1}{[K:\Q]}(-\sum_{v}\log||\frac{d\lambda}{2\lambda(\lambda-1)}||_v) \\
=&\frac{1}{[K:\Q]}(-\sum_{v|\infty}\log||\frac{d\lambda}{2\lambda(\lambda-1)}||_v-\sum_{v\text{ finite}}\log||\frac{d\lambda}{2\lambda(\lambda-1)}||_v)\\
=&\frac{1}{[K:\Q]}(-\sum_{v|\infty}\log||(\frac{dx}{2y})^{\otimes 2}||_v-\sum_{v\text{ not divides }2,\infty}\log||(\frac{dx}{2y})^{\otimes 2}||_v\\
&-\sum_{v|2}\log ||1/2||_v)\\
=&2h_F(E_{\lambda_0})+\frac{1}{[K:\Q]}\sum_{v|2}\log||(\frac{dx}{2y})^{\otimes 2}||_v-\log 2.
\end{split}
\end{equation}

Now we study $||(\frac{dx}{2y})^{\otimes 2}||_v$ given $v|2$. The sum $\frac{1}{[K:\Q]}\sum_{v|2}\log||(\frac{dx}{2y})^{\otimes 2}||_v$ does not change after extending $K$, hence we may assume that $\cE_{\lambda_0}$ over $\cO_v$ has the Deuring normal form $u^2+auw+u=w^3$ (see \cite{S} Appendix A Prop. 1.3 and the proof of Prop. 1.4 shows in the good reduction case, $a$ is a $v$-integer). An invariant differential generating $f_{*}\Omega_{\mathcal{E}_{\lambda_0}/\Spec \cO_K[\frac{1}{3}]}^{1}$ is $\frac{dw}{2u+aw+1}$. 

Because both $\frac{dw}{2u+aw+1}$ and $\frac{dx}{2y}$ are invariant differentials, we have $||\frac{dx}{2y}||_v=||\Delta_1/\Delta_2||_v^{\frac 1{12}}||\frac{dw}{2u+aw+1}||$, where $\Delta_1$ and $\Delta_2$ are the discriminant of the Deuring normal form and that of the Legendre form respectively. Since $E$ has good reduction, $||\Delta_1||_v= 1$ (see the proof of \emph{loc. cit.}). Hence $||\frac{dx}{2y}||_v=||\frac{dw}{2u+aw+b}||_v\cdot||1/16||^{1/12}_v=||2||^{-1/3}_v$.

Hence $\widehat{\deg}(T_{\lambda_0}X)= -2h_F(E_{\lambda_0})-\frac 2{3}\log 2+\log 2=-2h_F(E_{\lambda_0})+\frac{\log 2}{3}$.
\end{proof}

\begin{para}
As pointed out by Deligne (\cite{D}*{1.5}), the point $\frac{1+\sqrt{3}i}{2}$ corresponds to the elliptic curve with smallest height. Hence, our choice $\frac{1+\sqrt{3}i}{2}$ gives the largest $\widehat{\deg}(T_{\lambda_0}X)$ among those $\lambda_0$ such that $\lambda_0$ and $\lambda_0-1$ are units at every prime.
\end{para}

\begin{subsection}{The general case}
For the general case when $\lambda_0\in X(\bar{\Q})$, using a similar argument as in section \ref{sec_KS}, we have
\begin{equation}
\begin{split}
\frac{1}{[K:\mathbb{Q}]}(-\sum_{\sigma:K\rightarrow\mathbb{C}}\log||\frac d{d\lambda}||_{\sigma})&\leq -2h_F(E_{\lambda_0})+\frac{\log 2}3\\
&+\frac 1{[K:\Q]}\Big(\sum_{v\text{ finite}}\log^+||\lambda_0||_v+\log(|\Nm\lambda_0(\lambda_0-1)|)\Big)
\end{split}
\end{equation}
and equality holds if and only if $\lambda_0\in X(\bar{\Z}_2)$. As discussed in \ref{sec_TRinf}, the left hand side is the sum of the logarithms of our estimates of the radii of uniformizability at archimedean places.

We also need to modify the estimate of the radii at finite places in Lemma \ref{lem_lmg}. 
A possible estimate for $R_v$ is $p^{-\frac 1{p(p-1)}}\cdot \min\{||\lambda_0||_v,||\lambda_0-1||_v,1\}$. One explanation of the factor $\min\{||\lambda_0||_v,||\lambda_0-1||_v,1\}$ is that we cannot rule out the possibility that one has local monodromy at $0,1,\infty$ merely from the information of $p$-curvature at $v$.

Compared to the case when $\lambda_0\in X(\bar{\Z})$, our estimate for the sum of the logarithms of the archimedean radii increases by at most $\frac 1{[K:\Q]}(\sum_{v\text{ finite}}\log^+||\lambda_0||_v+\log(|\Nm\lambda_0(\lambda_0-1)|))$, while the estimate for the sum of logarithms of the radii at finite places becomes smaller by $\sum_v\max\{\log^+||\lambda_0^{-1}||_v,\log^+||(\lambda_0-1)^{-1}||_v\}.$ An explicit computation shows that the later is larger than the former. Hence the estimate for the product of the radii does not become larger than the case when 
$\lambda_0\in X(\bar{\Z})$.
\end{subsection}

\section{The affine elliptic curve case}\label{ec}
Let $X\subset \bA^2_{\Z}$ be the affine curve over $\Z$ defined by the equation $y^2=x(x-1)(x+1)$. The generic fiber $X_{\Q}$ is an elliptic curve (with $j$-invariant $1728$) minus its identity point. Given a vector bundle with connection over $X_K$, we will define the notion of vanishing $p$-curvature for \emph{all} finite places along the same lines as in section \ref{pcurv}. The main result of this section is:

\begin{theorem}\label{thm_ec}
Let $(M,\nabla)$ be a vector bundle with connection over $X_K$. Suppose that the $p$-curvatures of $(M,\nabla)$ vanish for all p. Then $(M,\nabla)$ is \'etale locally trivial.
\end{theorem}

\begin{rems}
This theorem cannot be deduced from applying Theorem \ref{thm_str} to the push-forward of $(M,\nabla)$ via some finite \'etale map from an open subvariety of the affine elliptic curve to $\bP^1_K-\{0,1,\infty\}$. Unlike the $\bP^1_K-\{0,1,\infty\}$ case, the conclusion here allows the existence of $(M,\nabla)$ with finite nontrivial monodromy. See section \ref{2isog}.
\end{rems}

\subsection{Formal horizontal sections and $p$-curvatures}
\begin{para}\label{ec_setup}
We fix $x_0=(0,0)\in X(\Z)$ and denote by $(x_0)_K$ and $(x_0)_{k_v}$ the images of $x_0$ in $X(K)$ and $X(k_v)$.
Let $y:X\rightarrow \bA^1_\Z$ be the projection to the $y$-coordinate. It is easy to check that this map is \'etale along $x_0$ and hence induces isomorphisms between the tangent spaces $T_{x_0}X\cong T_{0}\bA^1_\Z$ and between the formal schemes $\widehat{X_K}_{/(x_0)_K}\cong \widehat{\bA^1_{K}}_{/0}$.  In particular, we have an analytic section $s_v$ of the projection $y$ from $D(0,1)\subset \bA^1(K_v)$ to $X(K_v)$ such that $s_v(0)=x_0$ for any finite place $v$ by the lifting criterion for \'etale maps. By definition, the image $s_v(D(0,1))$ is the open rigid analytic disc in $X(K_v)$ which is the preimage of $(x_0)_{k_v}$ under the reduction map $X(K_v)\rightarrow X(k_v)$.

By choosing a trivialization of $M$ in some neighborhood of $(x_0)_K$, we can view a formal horizontal section $m$ of $(M,\nabla)$ around $(x_0)_K$ as a formal function in $\widehat{\cO_{X_K,(x_0)_K}}^r\cong \widehat{\cO_{\bA^1_K,0}}^r$, where $r$ is the rank of $M$. We denote $f\in \widehat{\cO_{\bA^1_K,0}}^r$ to be the image and the goal of this section is to prove that the formal power series $f$ is algebraic.
\end{para}

Let $U$ be $X-\{(0,1),(0,-1)\}$. It is a smooth scheme over $\Z$. Our chosen point $x_0$ is a $\Z$-point of $U$ and $s_v(D(0,1))\subset U(K_v)$.  For $v|p$ a finite place of $K$, we say that $(M,\nabla)$ has \emph{good reduction} at $v$ if $(M,\nabla)$ extends to a vector bundle with connection on $U_{\cO_v}$. Similar to Lemma \ref{lem_lmg}, we have:

\begin{lemma}\label{ec_Rfin}
Suppose that $(M,\nabla)$ has good reduction at $v$. If the $p$-curvature $\psi_p$ vanishes\footnote{This means $\psi_p\equiv0$ on $X_{\cO_v}\otimes \Z/p\Z$ as in section \ref{def_pcurv}.}, then the formal function $f$ is the germ of some meromorphic function on the disc $D(0, p^{-\frac 1 {p(p-1)}})\subset \bA^1$.
\end{lemma}
\begin{proof}
Let $(\cM,\nabla)$ be an extension of $(M,\nabla)$ over $X_{\cO_v}$.
Since $y$ is \'etale, the derivation $\frac \partial {\partial y}$ is regular over some Zariski open neighborhood $\bar{V}$ of $x_0\in X\otimes \Z/p\Z$. Let $V\subset X(K_v)$ be the preimage of $\bar{V}$ under reduction map. Since the $p$-curvature vanishes, we have $\nabla(\frac \partial {\partial y})^p(\cM|_V)\subset p\cM|_V$. Notice that $s_v(D(0,1))\subset V$. Then the proof of Lemma \ref{lem_lmg} shows the existence of horizontal sections of $M$ on $s_v(D(0, p^{-\frac 1 {p(p-1)}}))$. Via a local trivialization of $M$ and the isomorphism of formal neighborhoods of $x_0$ and $0$, we see that $f$ is meromorphic over $D(0, p^{-\frac 1 {p(p-1)}})$.
\end{proof}

This lemma motivates the following definition:
\begin{defn}\label{ec_allp}
We say that \emph{the $p$-curvatures of $(M,\nabla)$ vanish for all p} if
\begin{enumerate}
\item the $p$-curvature $\psi_p$ vanishes for all but finitely many $p$,
\item all formal horizontal sections around $x_0$, when viewed as formal functions in $\widehat{\cO_{\bA^1_K,0}}^r$, are the germs of some meromorphic functions on $D(0, p^{-\frac 1 {p(p-1)}})$ for all finite places $v$.
\end{enumerate}
\end{defn}
\begin{rems}
The second condition does not depend on the choice of local trivialization of $M$. Moreover, for each $v$, this condition remains the same if we replace the projection $y$ by any map $g: W_{\cO_v}\rightarrow \bA^1_{\cO_v}$ such that $W_{\cO_v}$ is a Zariski open neighborhood of $(x_0)_{\cO_v}$ in $X_{\cO_v}$ and that $g$ is \'etale.
\end{rems}

\subsection{Estimate at archimedean places and algebraicity}

Let $\sigma:K\rightarrow \C$ be an archimedean place.
Let $\phi: D(0,1)\rightarrow X(\C)$ be a uniformization map such that $\phi(0)=x_0$. We have the following lemma whose proof is the same as that of Lemma \ref{inf}:

\begin{lemma}
The $\sigma$-adic radius $R_\sigma$ (see Definition \ref{def_r}) of the formal functions $f$ in \ref{ec_setup} would be at least $|(y\circ \phi)'(0)|_\sigma$.
\end{lemma}

Let $t_0=\frac{1+i}2$. A direct manipulation of the definition shows $\lambda(t_0)=-1$, where $\lambda$ is defined in \ref{lambda}.
Let $F:D(0,1)\rightarrow \C-(\Z+t_0\Z)$ be a uniformization map such that $F(0)=\frac 12$.
\begin{lemma}\label{Er}(Eremenko) 
The derivative $|F'(0)|=2^{-3/2}\pi^{-3/2}\Gamma(1/4)^2=0.8346..$.
\end{lemma}

\begin{proof}
From \cite{E}*{Sec. 2}, we have $F'(0)=\frac {2^{5/2}}{B(1/4,1/4)}|(\lambda^{-1})'(i)|$\footnote{The choice of $\lambda$ there is different. We have $\lambda(i)=2$ here.}, where $B$ is Beta function. By Lemma \ref{lem_deri}, the Chowla-Selberg formula (\cite{SC}) 
\begin{equation}\label{CS}
|\eta(i)|=2^{-1}\pi^{-3/4}\Gamma(1/4),
\end{equation}
and $\theta_{00}^4(i)=2\theta_{01}^4(i)=2\theta_{10}^4(i)$,
we have
$$|(\lambda^{-1})'(i)|=|\pi i (\frac {\theta_{01}(i)\theta_{10}(i)}{\theta_{00}(i)})^4|=\pi |\eta(i)|^4=\frac{\Gamma(1/4)^4}{2^4\pi^2}.$$
We obtain the desired formula by noticing that $B(1/4,1/4)=\pi^{-1/2}\Gamma(1/4)^2.$
\end{proof}

\begin{lemma}\label{unimap}
Let $\alpha$ be the constant $2(-\pi^2\theta_{01}^4(t_0))^{3/2}$ and $\wp$ be the Weierstrass-$\wp$ function. We have $y\circ \phi=\alpha^{-1}\wp'\circ F$, up to some rotation on $D(0,1)$. 
\end{lemma}
\begin{proof}
The map $g\colonequals (\wp,\wp')$ maps $\C-(\Z+t_0\Z)$ to the affine curve $u^2=4v^3-g_2(t_0)v-g_3(t_0)$. Let $s$ be the isomorphism from this affine curve to $X(\C)$ given by (\ref{isom}). Since both $s\circ g(1/2)$ and $x_0$ are the unique point fixed by the four automorphisms of $X(\C)$, we have $s\circ g(1/2)=x_0$. Hence $s\circ g\circ F(0)=x_0=\phi(0)$ and then the uniformizations $s\circ g\circ F$ and $\phi$ are the same up to some rotation. 
We have $y\circ \phi=y\circ s\circ g\circ F=\alpha^{-1}\wp'\circ F$ by (\ref{isom}).
\end{proof}

\begin{proposition}\label{ec_Rinf}
The $\sigma$-adic radius $R_\sigma^{\frac{[K:\Q]}{[K_\sigma:\R]}}\geq2^{-5/2}\pi^{-2}\Gamma(1/4)^4=3.0949\cdots.$
\end{proposition}
\begin{proof}
Differentiate both sides of $(\wp'(z))^2=4(\wp(z))^3-g_2(t_0)\wp(z)-g_3(t_0)$,
we have
$$\wp''(1/2)=6\wp(1/2)^2-g_2(t_0)/2=-g_2(t_0)/2,$$
where the second equality follows from that
$$\wp(1/2)=\pi^{2}(\theta_{00}^{4}(t_{0})+\theta_{01}^{4}(t_{0}))/3=\pi^{2}\theta_{01}^{4}(t_{0})(\lambda(t_0)+1)/3=0.$$
By Lemma \ref{lem_deri} and $\theta_{00}^4(t_0)=-\theta_{01}^4(t_0)=\theta_{10}^4(t_0)/2$, we have
$$|g_2(t_0)|=\frac {4\pi^4}{3}\cdot \frac 12 |\theta_{00}^8(t_0)+\theta_{01}^8(t_0)+\theta_{10}^8(t_0)|=4\pi^4|\theta_{01}^8(t_0)|.$$
Then by Lemma \ref{unimap} the absolute value of the derivative of $y\circ \phi$ at $0$ would be
\begin{equation}
\begin{split}
|\alpha^{-1}\wp''(1/2)\cdot F'(0)|&=2^{-1}\pi^{-3}|\theta_{01}(t_0)|^{-6}\cdot 2\pi^4|\theta_{01}(t_0)|^8\cdot |F'(0)|\\
&=\pi|\theta_{01}(t_0)|^2\cdot 2^{-3/2}\pi^{-3/2}\Gamma(1/4)^2\text{ (by Lemma \ref{Er})}\\
&=2\pi\cdot 2^{-2}\pi^{-3/2}\Gamma(1/4)^2\cdot 2^{-3/2}\pi^{-3/2}\Gamma(1/4)^2\\
&=2^{-5/2}\pi^{-2}\Gamma(1/4)^4=3.0949\cdots,
\end{split}
\end{equation}
where the third equality follows from
$$|\theta_{01}(t_0)|=2^{-1/12}|\theta_{00}(t_0)\theta_{01}(t_0)\theta_{10}(t_0)|^{1/24}=2^{1/4}|\eta(t_0)|=2^{1/2}|\eta(i)|,$$
and (\ref{CS}).
\end{proof}

\begin{proof}[Proof of Theorem \ref{thm_ec}]
By Proposition \ref{ec_Rinf}, we have $\prod_{v|\infty}R_v\geq 3.0949\cdots.$
By Definition \ref{ec_allp}, we have $\log(\prod_{v\nmid \infty} R_v)\geq -\sum_p\frac{\log p}{p(p-1)}=-0.761196\cdots.$ Hence $$\log(\prod_v R_v)\geq \log 3.0949\cdots-0.761196\cdots=0.3685\cdots>0.$$
We conclude by applying Theorem \ref{andre}. 
\end{proof}

\section{Examples}\label{ec_eg}
In this section, we first give an example of $(M,\nabla)$ with $p$-curvature vanishing for all $\fp$ but with nontrivial global monodromy over the affine elliptic curve in section \ref{ec}. Then we discuss a variant of our main theorems with $X$ being the affine line minus all $4$-th roots of unity. 

\subsection{An example with vanishing $p$-curvature for all $\fp$ and nontrivial $G_{\gal}$}\label{2isog}
Let $K$ be $\Q(\sqrt{-1})$, $X\subset \bA^2_{\Z}$ be the affine curve defined by $y^2=x(x-1)(x+1)$, $E$ be the elliptic curve defined as the compactification of $X_K$, and $f:E\rightarrow E$ be a degree two self isogeny of $E$. We will also use $f$ to denote the restriction of $f$ to $X_K\backslash \{P\}$, where $P$ is the non-identity element in the kernel of $f$. 

Let $(M,\nabla)$ be $f_*(\cO_{X_K\backslash \{P\}},d)$. By definition, $G_{\gal}$ is $\Z/2\Z$. 
\begin{proposition}
The $p$-curvature of $(M,\nabla)$ vanishes for all $\fp$. 
\end{proposition}
\begin{proof}
Notice that $f$ extends to a degree two \'etale cover from $E$ to $E$ over $\Z[\frac i2]$. Then
for $\fp\nmid 2$, the $p$-curvature of $(M,\nabla)$ coincides with that of $f^*(M,\nabla)$ by the fact that $p$-curvatures remain the same under \'etale pull back\footnote{Because $p\neq 2$ is unramified in $K$ and $(M,\nabla)$ has good reduction at $\fp$, the notion of $p$-curvature here is classical.}. Hence the $p$-curvature of $(M,\nabla)$ vanishes as $f^*(M,\nabla)$ is trivial.

For $\fp|2$, we write $(M,\nabla)$ out explicitly. Without loss of generality, we may assume that $f$ from the curve $y^2=x(x-1)(x+1)$ to the curve $s^2=t(t-1)(t+1)$ is given by $t=-\frac i2(x-\frac 1x)$ and $s=\frac{1+i}4\frac yx(x+\frac 1x)$. Locally around $(t,s)=(0,0)$, $1,x$ is an $\cO_{X_K}$ basis of $f_*\cO_{X_K}$ and this basis gives rise to a natural Zariski local extension of $(M,\nabla)$ over $X_{\cO_{\fp}}$. Direct calculation shows that 
$$\nabla(1)=0,\nabla(x)=\frac{2s}{(t^2-1)(3t^2-1)}ds+\frac{2st(1+2i)}{(t^2-1)(3t^2-1)}xds.$$ Therefore, $\nabla(f_1+f_2x) \equiv df_1+xdf_2\, (\mathrm{mod}\,2)$ and the $p$-curvature of $(M,\nabla)$ vanishes.
\end{proof}

\begin{rems}
In the above proof, we show that $(M,\nabla)$ has all $p$-curvatures vanishing in the strict sense: there is an extension of $(M,\nabla)$ over $X_{\cO_K}$ such that its $p$-curvatures are all vanishing. However, given the argument for $\fp\nmid 2$, in order to to apply Theorem \ref{thm_ec}, we do not need to construct an extension of $(M,\nabla)$ but only need to check that $x$, locally as a formal power series of $s$, converges on $D(0,2^{-1/2})$ for $v|2$. This is not hard to see: $x$, as a power series of $t$, converges when $|t|_v<|2|_v$; and $t$, as a power series of $s$,  converges when $|s|_v<|2|_v^{1/2}$ and the image of $|s|_v<|2|_v^{1/2}$ is contained in $|t|_v<|2|_v$.
\end{rems}

\subsection{A variant of the main theorems}
In this section, we will prove a variant of the main theorems when $X=\bA^1_{\Q}-\{\pm 1,\pm i\}$. Similar to Theorem \ref{thm_ec}, the conclusion is that $(M,\nabla)$ has finite monodromy and we give an example with nontrivial finite monodromy. 

To define the local convergence conditions for bad primes, we take $x_0=0$.

\begin{proposition}\label{5pts}
Let $(M,\nabla)$ be a vector bundle with connection over $X$ with $p$-curvature vanishes for all $\fp$. We further assume that the formal horizontal sections around $x_0$ converge over $D(x_0,1)$ for $v|15$. Then $(M,\nabla)$ is \'etale locally trivial.
\end{proposition}

\begin{proof}
By Lemma \ref{Er}, we have $R_\infty\geq 2\cdot 0.8346\cdots$. By the assumptions on finite places, we have $\log(\prod_{v\nmid \infty}R_v)\geq -\sum_{p\neq 3,5}\frac{\log p}{p(p-1)}=-0.4976\cdots$. Then we conclude by applying Theorem \ref{andre}.
\end{proof}

\begin{example}
Let $s$ be $(1-x^4)^{1/2}$. It is the solution of the differential equation $\frac{ds}{dx}=\frac{-2x^3}{1-x^4}$. Consider the connection on $\cO_X$ given by $\nabla(f)=df+\frac{2x^3}{1-x^4}dx$. It has $p$-curvature vanishing for all $p$: $\nabla(f)\equiv df$ (mod $2$) and $\nabla(f)\equiv df+(p+1)\frac{2x^3}{1-x^4}dx$ (mod $p$) with solution $s\equiv(1-x^4)^{(p+1)/2}$ (mod $p$) when $p\neq 2$. 
In conclusion, $(\cO_X,\nabla)$ satisfies the assumptions in the above proposition while it has nontrivial monodromy of order two.
\end{example}

\begin{rems}
If we replace our assumption by similar conditions on generic radii, the above example shows that one could have order two local monodromy around $\pm 1,\pm i$. The reason is \cite{BS}*{III eqn. (3)} does not hold in this situation and a modification of their argument would show that an order two local monodromy is possible.
\end{rems}

\end{document}